 \documentclass{article}
 \usepackage{amsmath}
 \usepackage{amssymb}
 \usepackage{latexsym}
 \usepackage{amsthm}
 \usepackage{mathrsfs}
 \usepackage{wasysym}
 \usepackage{fancyhdr}
 \usepackage{xcolor}
 \usepackage{amsfonts}
 \usepackage{amssymb}
 \usepackage{amsmath}
 \usepackage{epsfig}
 \usepackage{epstopdf}
 \usepackage[colorlinks=true,linkcolor=blue,citecolor=blue,%
	urlcolor=blue]{hyperref}

 \newtheorem{thm}{Theorem}[section]
 
 \newtheorem{lemma}[thm]{Lemma}

 \newtheorem{obs}[thm]{Observation}
 
 \newtheorem{question}[thm]{Question}

 \theoremstyle{definition}
 \newtheorem*{definition}{Definition}

 \newtheorem*{basis}{Basis}
 \newtheorem*{inductionstep}{Induction step}
 \theoremstyle{definition}
 \newtheorem{stage-oe}{(\ref*{orieven}) Stage}
 \newtheorem{stage-oo}{(\ref*{oriodd}) Stage}
 \newtheorem{stage-ne}{(\ref*{nonorieven}) Stage}
 \newtheorem{stage-no}{(\ref*{nonoriodd}) Stage}

  \let\Si\Sigma

  \let\vP=P
  \let\vQ=Q
  \let\vX=X
  \let\vY=Y

  \long\def\ignore#1{}

  \def\iv{^{-1}}
  \let\eugen\gamma 
  \def\di(#1,#2){{\fam1 di}(#1#2)}
  \def\diq(#1,#2,#3){{\fam1 di}(#1#2,#3)}

 \let\Q=Q
 \def\nQ{\widetilde{Q}}
 \def\nT{\widetilde{T}}
 \def\sD{\mathcal{D}}
 \let\ab\allowbreak
 \def\pr{${}'$}

\begin{document}
  \title{Minimal quadrangulations of surfaces}
  \author{
  Wenzhong Liu\thanks{Department of Mathematics, Nanjing University of
  Aeronautics and Astronautics, Nanjing 210016, China. Email:
 \texttt{wzhliu7502@nuaa.edu.cn}.},\;
  M. N. Ellingham\thanks{
 Department of Mathematics, Vanderbilt University, Nashville, TN 37240, USA,
 Email: \texttt{mark.ellingham@vanderbilt.edu}.
 Partially supported by Simons Foundation award no.~429625.}\,
 and
  Dong Ye\thanks{Department of Mathematical Sciences, Middle Tennessee
 State University, Murfreesboro, TN 37132, USA,
 Email: \texttt{dong.ye@mtsu.edu}.
 Partially supported by Simons Foundation award no.~359516.}\;
 }
 \date{24 June 2021}
 \maketitle

 \vspace{-10pt}
 \begin{center}
 In memory of Nora Hartsfield
 \end{center}

 \begin{abstract}
 A
\emph{quadrangular} embedding of a graph in a surface $\Si$, also known
as a \emph{quadrangulation of $\Si$}, is a cellular embedding in which
every face is bounded by a $4$-cycle.  A quadrangulation of $\Si$ is
\emph{minimal} if there is no quadrangular embedding of a
(simple) graph of smaller order in $\Si$.
 In this paper we determine $n(\Si)$, the order of a minimal
quadrangulation of a surface $\Si$, for all surfaces, both orientable
and nonorientable.
 Letting $S_0$ denote the sphere and $N_2$ the Klein bottle, we prove
that  $n(S_0)=4, n(N_2)=6$, and $n(\Si)=\lceil (5+\sqrt{25-16
\chi(\Si)})/2\rceil$ for all other surfaces $\Si$, where $\chi(\Si)$ is
the Euler characteristic.
 Our proofs use a `diagonal technique', introduced by Hartsfield in
1994.  We explain the general features of this method.

   \medskip


 \end{abstract}

 \section{Introduction}

  All graphs considered in this paper are simple.
  Let $G$ be a graph with vertex set $V(G)$ and edge set $E(G)$. For
convenience, we use $E_{k}$ to denote a subset of $E(G)$ with exactly
$k$ edges.
 A \emph{surface} is a
connected compact 2-manifold without boundary.
 The orientable surface of genus $g$ is denoted $S_g$, and the
  nonorientable surface of genus $q$ is denoted $N_q$.
  The Euler characteristic of a surface $\Si$ is denoted $\chi(\Si)$,
 which is $2-2g$ for $S_g$, and $2-q$ for $N_q$.
 The \emph{Euler genus} of $\Si$ is defined as $\eugen(\Si) = 2
- \chi(\Si)$.

  An embedding of a graph in a surface $\Si$ is {\em cellular} if every face of the embedding is homeomorphic to an open disc.
 All embeddings considered in the paper  are  cellular.
 An embedding is  {\em quadrangular}, or  a {\em quadrangulation} of  $\Si$,
  if every face is bounded by a $4$-cycle.
 A face bounded by a $4$-cycle is a \emph{quadrangle}, or to
use a shorter word, a \emph{square}.
 A quadrangulation of  $\Si$ is \emph{minimal} if there is no quadrangular embedding of a
 graph of smaller order in $\Si$.
 Similarly,  a {\em triangulation} of  $\Si$ is an embedding of
 a graph in $\Si$ such that every face is bounded by a $3$-cycle.
  A triangulation of $\Si$ is \emph{minimal} if  there is no triangular embedding of a
 graph  with a smaller order in $\Si$.

 Thomassen \cite{TH, THn} showed that given a graph $G$ and an integer
$k$, it is NP-complete to determine whether $G$ has an embedding in a
surface of orientable (or nonorientable) genus at most $k$.
 In other words, determining the minimum genus of an embedding of $G$ is
difficult.
 A minimum genus embedding of a graph maximizes the number of faces over
all its embeddings, and hence often has many small faces.
 Triangular embeddings of a given $G$ are always minimum
genus embeddings.
 However, we can also consider triangular embeddings from the
perspective of surfaces.  Peschl (see \cite{JR}) asked how many vertices
a triangulation of a given surface $\Si$ must have.

 A triangular embedding of a complete graph $K_n$ in a given surface
$\Si$ is both a minimum genus embedding of $K_n$, and a minimum order
triangulation of $\Si$.  Such embeddings played a key role in the proof
of the Map Color Theorem (see \cite{RI}).  These embeddings were
generalized in two ways.  For some values of $n$, there is no triangular
embedding of $K_n$, so to determine the minimum genus of $K_n$,
embeddings were used where most, but not all, of the faces are
triangular (again see \cite{RI}).  For most surfaces $\Si$, there is no
complete graph with a triangular embedding in $\Si$, so to find minimal
triangulations of $\Si$ we must use graphs close to complete graphs.
Ringel \cite{RI55} did this for nonorientable surfaces, and Jungerman
and Ringel \cite{JR} for orientable surfaces.

 Quadrangular embeddings are also of interest.  For bipartite graphs,
quadrangular embeddings have minimum genus.  For non-bipartite graphs,
quadrangular embeddings have minimum genus over all embeddings with face
degrees $4$ or more, or with even face degrees.
 Ringel \cite{RI65a, RI65b} determined the minimum genus of complete
bipartite graphs, which used quadrangular embeddings in many cases;
Bouchet \cite{BO} provided a simpler proof.
 Quadrangular embeddings of nearly complete bipartite graphs and graphs
obtained from some graph operations were studied in \cite{CR, MPP, PI1,
PI2, WH70, WH}.
 Hartsfield and Ringel \cite{HR1, HR2}
 found quadrangular embeddings of the complete graph $K_n$ in
orientable surfaces for $n \equiv 5$ (mod $8$), and in nonorientable
surfaces for $n \equiv 1$ (mod $4$) and $n \ne 1,5$.
 They also found both orientable and nonorientable quadrangular
embeddings of the general octahedral graph
 $O_{2n}$, obtained by removing a perfect matching from $K_{2n}$.
 Using the `diagonal technique', which we
discuss in more detail below, Hartsfield \cite{HA} outlined a proof
that a complete multipartite graph $K_{n_1, n_2,  \dots, n_t}$ with an
even number of edges, other than $ K_5 $ and $K_{1, m, n}$, has a
quadrangular embedding in a nonorientable surface.  This includes
nonorientable quadrangular embeddings of $K_n$ when $n \equiv 0$ (mod
$4$).
 Korzhik and Voss \cite{KO, KV} constructed exponentially many
nonisomorphic  quadrangular embeddings of  the complete graph
$K_{8s+5}$.

 Recently, the authors and others \cite{LEYZ} determined the minimum
genus of an embedding of $K_n$ with even face degrees.
 (Lawrencenko, Chen and Yang, a
subset of the authors of \cite{LEYZ}, also have alternative current
graph proofs \cite{LCY18} of some of these results, although some
modification of the index $2$ current graphs is required.)
 This completed the proof of the Even Map Color Theorem, a strengthening
of the Map Color Theorem for embeddings with even face degrees, and
included a complete characterization of when $K_n$ has a quadrangular
embedding.

 \begin{thm} [\cite{ HR1, HR2, LEYZ}]  \label{complete}
  The complete graph $K_n$ has a quadrangular embedding in an orientable
 surface if and only if $n \equiv 0$ or $5$ (mod $8$), and in a
 nonorientable surface if and only if $n \equiv 0$ or $1 (mod~4)$ and $n
 \ne 1, 5$.
  \end{thm}

  The quadrangular embeddings of the complete graphs $K_n$ and the
general octahedral graphs $O_{2n}$ given in \cite{HR1, HR2, LEYZ}  are
all minimal quadrangulations of surfaces.
Other prior results on minimal quadrangulations, for which we provide
details later in this section, appear in \cite{HR1, LA, LEYZ}.
 The purpose of this paper is to construct, and hence determine the
order of, minimal quadrangulations for all surfaces.
 Our main results are as
follows; Theorem \ref{construction} provides the embeddings needed to
prove Theorems \ref{order} and \ref{face-simple}.

 \begin{thm}  \label{construction}

 Let $(n,t)$ be a pair of integers with $n \ge 4$
and $0 \le t \le n-4$.

 If $t \equiv \frac12 n(n-5)$ (mod $4$)
 then there is an orientable
quadrangular embedding of an $n$-vertex graph with $\binom{n}{2}-t$
edges.  There is also a quadrangulation of $S_0$ for $(n,t)=(4,2)$.

 If $t \equiv \frac12 n(n-5)$
(mod $2$)
 then there is a nonorientable
quadrangular embedding of an $n$-vertex graph with $\binom{n}{2}-t$
edges, unless $(n,t) = (5,0)$, in which case no such embedding exists.
 There is also a quadrangulation of $N_2$ for $(n,t)=(6,3)$.

 \end{thm}

 \begin{thm}  \label{order}
  Let $\Si$ be a surface with Euler characteristic
$\chi(\Si)$ and Euler genus $\eugen(\Si)$.  Let $n(\Si)$ be the number of vertices of a
minimal quadrangulation of $\Si$.
 If $\Si \ne S_{0}$ and $N_{2}$, then
 \begin{align*}
 n(\Si) = \left\lceil \frac{5+\sqrt{25-16\chi(\Si)}}{2} \right\rceil
 = \left\lceil \frac{5+\sqrt{16\eugen(\Si)-7}}{2} \right\rceil.
 \end{align*}
 Moreover,  $n(S_{0})=4$ and $n(N_{2})=6$.
 \end{thm}

 An embedding is \emph{face-simple} if its dual is simple, i.e.,
two face boundaries share at most one edge.  We can strengthen Theorem
\ref{order} slightly in the orientable case.

 \begin{thm}\label{face-simple}
 Let $n'(\Si)$ be the minimum number of vertices of a face-simple
quadrangular embedding of a simple graph in $\Si$.  Then $n'(S_0) = 8$
and $n'(S_g) = n(S_g)$ for all $g \ge 1$.
 \end{thm}

 We show in Section \ref{maintheorems} that all quadrangulations given
in Theorem \ref{construction} are minimal, and that this proves
Theorems \ref{order} and \ref{face-simple}.
 The main tool used to prove Theorem \ref{construction} is an approach
due to Hartsfield, which we call the `diagonal technique' and describe
in Section \ref{diagonal}.  As we explain there, Hartsfield wrote two
papers (one published, one not) using this idea, but her papers did not
contain complete proofs.  One of the contributions of this paper is to
provide an explicit overview of how the diagonal technique works, and to
demonstrate the rigorous use of this method.
 The actual proof of Theorem \ref{construction} is in Section
\ref{constructionproof}, divided into orientable and nonorientable
cases.
 Section \ref{final} gives some final remarks regarding Theorem
\ref{face-simple}.



 %
 Prior results on minimal quadrangulations proved some special cases of
Theorem \ref{construction}, constructing quadrangulations with $n$
vertices and $\binom{n}2-t$ edges for suitable $t$.
 Theorem \ref{complete} deals with the case $t=0$, and Hartsfield and
Ringel's results on octahedral graphs \cite{HR1, HR2} deal with the case
where $n$ is even and $t=n/2$.  They also proved the orientable case
when $n$ is even and $t=n/2+4$ \cite[Section 6]{HR1}.
 Lawrencenko \cite{LA} used a result originally due to White \cite{WH},
which can also be proved using Craft's graphical surface technique
\cite{CR}, to prove the orientable cases where $n$ is even and $n/2 \le
t \le n-4$.
 Liu et al.~\cite[Corollary 7.2]{LEYZ} extended this idea to prove the
nonorientable cases where $n$ is even, $n/2 \le t \le n-4$, and $t
\equiv \frac12 n(n-5)$ (mod $4$) (but not $t \equiv 2+\frac12 n(n-5)$
(mod $4$)).
 Moreover, \cite[Corollary 7.3]{LEYZ} handles all cases (orientable and
nonorientable) where $8 \,|\, n$ and $16 \,|\, t$, and \cite[Corollary
7.4]{LEYZ} handles all nonorientable cases where $t=n-4$ and all
orientable cases where $n-6 \le t \le n-4$.

 %
 %
 %
 %
 %
 %
 %
 %

 %
 Note that Magajna, Mohar and Pisanski \cite{MMP} solved a problem
related to minimal quadrangulations, by showing that for every surface
$\Si$ the minimum number of vertices of a bipartite graph with a
quadrangular embedding in $\Si$ is $\lceil 4 + \sqrt{16-8\chi(\Si)}
\rceil$.

 \section{Relationship between the main theorems}\label{maintheorems}
 
 In this section we show that the quadrangulations described in
Theorem \ref{construction} are minimal, and that Theorem
\ref{construction} implies Theorems \ref{order} and
\ref{face-simple}.
 Suppose we have a quadrangular embedding in a surface $\Si$ with $n$
vertices, $m = \binom{n}{2}-t$ edges, and $r$ faces. Counting sides of
edges gives $2m=4r$ or $r = \frac12 m$, and so from Euler's formula
$\chi(\Si) = n-m+r = n-\frac12 m$.  Hence
 \begin{align}\label{quadeq}
 -2\chi(\Si) =
m - 2n = \binom{n}{2} -t - 2n = {\textstyle\frac12} n(n-5) - t.
 \end{align}
 We have a sufficient condition for such an embedding to be minimal.

 \begin{lemma}[{\cite[Lemma 7.1]{LEYZ}}]\label{minquad}
 Suppose that $n \ge 5$, $0 \le t \le n-4$, and $L$ is a graph with $n$
vertices and $m=\binom{n}{2}-t$ edges.
 Then any quadrangular embedding of $L$ is minimal.
 \end{lemma}

 We now consider properties of the right-hand side of (\ref{quadeq}).

 \begin{lemma}\label{fprop}
 Let $f(x) = \frac12 x(x-5)$, defined on $[3,\infty)$.

 \noindent
 (a) If $n \ge 3$ is an integer, then $f(n)$ is an integer.

 \noindent
 (b) For every real number $y > f(3)=-3$ there exists a unique pair
$(n,t)$ where $n \ge 4$ is an integer, $0 \le t < n-3$, and $y =
f(n)-t$.
 Moreover, $\displaystyle n = \lceil f\iv(y) \rceil
 = \left\lceil \frac{5 + \sqrt{25+8y}}{2} \right\rceil$.

 \noindent
 (c) Therefore, if $k \ge -2$ is an integer there exists a unique pair
of integers $(n,t)$ where $n \ge 4$, $0 \le t \le n-4$, and $k
= f(n)-t$ (or
 $\frac12 n(n-5) = t+k$).  Moreover,
 $\displaystyle n = \left\lceil \frac{5 + \sqrt{25+8k}}{2}
\right\rceil$.

 \end{lemma}

 \begin{proof} For (a), if $n$ is an integer then $2f(n)=n(n-5)$ is
even, so $f(n)$ is an integer.  For (b), since $f'(x) = x - \frac52 > 0$
on $[3, \infty)$, $f$ is strictly increasing, and clearly $f(x) \to
\infty$ as $x \to \infty$.  Therefore, every $y > f(3)$ lies in
a unique interval $(f(n-1), f(n)] = (f(n)-(n-3), f(n)]$
for some integer $n \ge 4$, so that $y = f(n)-t$ where $0 \le t < n-3$. 
Moreover, $n-1 < f\iv(y) \le n$ so that $n = \lceil f\iv(y) \rceil$, and
$f\iv(y)$ is found by the quadratic formula.  Now (c) follows from (a)
and (b).
 \end{proof}

 \begin{proof}[Proof that Theorem \ref{construction} implies Theorem
\ref{order}]
 
 A quadrangulation has $n \ge 4$ vertices, so the special
case $(n,t)=(4,2)$ and regular case
$(n,t)=(4,0)$ of Theorem \ref{construction} verify Theorem
\ref{order} for $\Si = S_0$ and $N_1$, respectively.  
 By equation (\ref{quadeq}) a quadrangulation of $N_2$ must
have
 $\frac12 n(n-5) = -2\chi(N_2) + t = t \ge 0$, so $n \ge
5$, and if $n=5$ then $t=0$.
 By Theorem \ref{construction} there is no quadrangulation of $N_2$ for
$(n,t)=(5,0)$, so the one for $(n,t)=(6,3)$ is minimal, verifying
Theorem \ref{order} for $\Si = N_2$.

 So assume $\Si \ne N_2$ is a surface with $\chi = \chi(\Si) \le 0$. 
Applying Lemma \ref{fprop}(c) with $k = -2\chi \ge 0$, there are
integers $n$ and $t$ that satisfy equation (\ref{quadeq}) (or
 $\frac12 n(n-5) = t-2\chi$)
 and $0 \le t \le n-4$.  Moreover,
 $\displaystyle n = \left\lceil \frac{5 + \sqrt{25-16\chi}}{2}
\right\rceil \ge 5$.
 If $\Si$ is orientable then $\chi$ is even, so
 $\frac12 n(n-5)=t-2\chi
\equiv t$ (mod $4$).  Thus, the first part of Theorem \ref{construction} gives
an orientable
quadrangulation with $n$ vertices and $\binom{n}2-t$ edges.  This is embedded in $\Si$ by (\ref{quadeq}), minimal by Lemma
\ref{minquad}, and has order $n$ satisfying Theorem \ref{order}.  We use
the second part of Theorem \ref{construction} in a similar way when
$\Si$ is nonorientable.
 \end{proof}

 We also verify Theorem \ref{face-simple}, using the following.

 \begin{obs}[{\cite[Observation 3.4]{LEYZ}}]\label{fsquad}
 Suppose $\Phi$ is a quadrangular embedding of a simple connected graph
with minimum degree at least $3$.
 If $\Phi$ is not face-simple then it contains two squares of the form
$uvwx$ and $uvxw$. Thus, if $\Phi$ is orientable or bipartite then it is
face-simple.
 \end{obs}

 \begin{proof}[Proof of Theorem \ref{face-simple}]
 Suppose $\Phi$ is a face-simple quadrangulation of $S_0$ with $n$
vertices, $m$ edges and $r$ faces.  By Euler's formula $m = 2n-4$ and $r
= \frac12 m = n-2$.
 Let $H$ be the underlying graph of the dual $\Phi^*$.  Then $r = |V(H)|
\ge 6$, because $H$ is a $4$-regular simple graph that is planar and
hence not $K_5$.  Hence $n = r+2 \ge 8$, and the usual quadrangular
embedding of the cube (whose dual is the octahedron, which is simple)
has $n=8$.  Thus, $n'(S_0) = 8$.

 For $g \ge 1$, we know from above that there is a minimal
quadrangulation $\Phi$ of $S_g$ as in Lemma \ref{minquad}.  Since the
underlying graph is obtained from a complete graph by deleting at most
$n-4$ edges, the edge-connectivity, and hence also the minimum
degree, is at least $(n-1)-(n-4)=3$, and so $\Phi$ is face-simple by
Observation \ref{fsquad}.  Thus, $n'(S_g) = n(S_g)$.
 \end{proof}

 \section{Hartsfield's diagonal technique}\label{diagonal}

 In this section, we describe the operations that form part of a method 
introduced by Hartsfield \cite{HA, HA94p}, which we will call the \emph{diagonal technique}.
 This technique applies specifically to constructing embeddings that are
quadrangular, or where most faces are squares (quadrangles).

 This technique was used by Hartsfield in
two papers, one published \cite{HA} and one not \cite{HA94p}.
 In \cite{HA} Hartsfield claimed to construct nonorientable quadrangular
embeddings of almost all complete multipartite graphs with an even
number of edges, except for $K_5$ and complete tripartite graphs
$K_{1,m,n}$.
 This included complete graphs $K_n$ with $n \equiv 0$ (mod $4$) (for $n
\equiv 1$ (mod $4$) Hartsfield used embeddings from \cite{HR2}).
 In \cite{HA94p} Hartsfield claimed to construct both orientable and
nonorientable minimum genus embeddings with all faces of degree
$4$ or more for $K_n$, $n \ge 4$.
 Unfortunately Hartsfield did not provide rigorous proofs in either of
these papers.  She illustrated the proof ideas with small examples, and
seemed to assume that it was clear how to generalize these.
 But she did not provide an explicit overview of how the constructions
are supposed to work and so from her papers it is hard to see how to extend the
small examples.
 Sadly, Hartsfield died in 2011, so she cannot provide
further elucidation.  But we have distilled the main ideas from what she
wrote, and we provide an outline at the end of this section, after
defining necessary concepts and operations.

 We hope that rigorous versions of Hartsfield's proofs will appear
eventually.  Lawrencenko et al.~\cite{LCYHp} are preparing a paper that
gives alternative proofs for the main result of \cite{LEYZ}, which
determined the minimum genus for orientable and nonorientable embeddings
with all faces of even degree, and with all faces of degree at least
$4$, for complete graphs.  This includes Theorem \ref{complete} as a
special case.  These alternative proofs are based both on current
graphs as in \cite{LCY18}, and on Hartsfield's proof
outlines from \cite{HA94p} which use the diagonal technique.

 As a byproduct, our results in this paper also provide a proof of
Theorem \ref{complete} using Hartsfield's diagonal technique. 
 However, since our goal is the construction of minimal
quadrangulations, rather than embeddings of complete graphs with minimum
genus subject to all faces having degree at least $4$, our proof
differs significantly from those in \cite{HA, HA94p, LCYHp}.
 We add two vertices at a time, rather than eight, and we use additional
operations (handle additions of Type III and crosscap additions; see
below).

 We now introduce some definitions that we will need to
implement the diagonal technique.
  Let $\Phi$  be a quadrangulation of a surface $\Si$.
  Every face of $\Phi$ is a square, bounded
by a $4$-cycle.
  We describe squares by listing their vertices in order around
the boundary.  For an orientable embedding, we always list the vertices
in clockwise order.
 
 Two nonadjacent vertices $v_{i}$ and $v_{j}$ of a square
form a {\em diagonal}, denoted by
$\di(v_{i},v_{j})$.
 The square is called the \emph{underlying square} of $\di(v_{i},v_{j})$.
 Clearly, each diagonal depends on an underlying square and this
underlying square may be not unique.
 For example, in Figure 1, $\di(v_1,v_3)$ has three
different underlying squares. When it will not cause
confusion,
 we pinpoint a diagonal but do not explicitly
mention its underlying square.
 If we do wish to indicate the underlying square, we write
$\diq(v_i,v_j,v_i v_k v_j v_\ell)$. 
 A \emph{diagonal set} is a set of  diagonals
that have different underlying squares.  A diagonal set  is
\emph{full} if it contains at least one diagonal incident
with every vertex, \emph{perfect} if it contains exactly one
diagonal incident with every vertex, and \emph{$v_i$-nearly-perfect} if
it contains no diagonal incident with $v_i$ and exactly one diagonal
incident with every vertex not equal to $v_i$.

 \begin{figure}[!hbtp]\refstepcounter{figure}\label{planar}
  \begin{center}
  \includegraphics[scale=0.7]{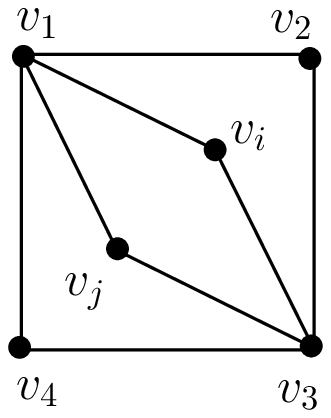} \\
  {Figure~\ref{planar}: Disc addition}
  \end{center}
  \end{figure}

 \medskip
 \noindent {\bf Operation 1: Disc  addition}
 \medskip

  Let $f=v_1v_2v_3v_4$ be a square of the quadrangulation $\Phi$ of
a surface $\Si$. Add two new vertices
 $v_{i}$ and $v_{j}$ into the interior of the square $f$,
 and then join $v_{i}$ and $v_{j}$  to the diagonal
$\di(v_{1},v_{3})$ of $f$
  by  four new edges $v_{1}v_{i}$, $v_{1}v_{j}$, $v_{3}v_{i}$ and $v_{3}v_{j}$.
  The square $f$ is divided into  three new squares $f_1=v_{1}v_{2}v_{3}v_{i}$, $f_2=v_{1}v_{i}v_{3}v_{j}$ and $f_3=v_{1}v_{j}v_{3}v_{4}$.
   All the new squares $f_i$ with $i\in \{1,2, 3\}$ have $\di(v_1,v_3)$
as one of their diagonals. This operation is
called a {\em disc addition} (Hartsfield called this a
\emph{planar addition}).
  Applying a disc addition to a square of $\Phi$ generates a new
quadrangulation of the same surface $\Si$, with the
same genus.  See  Figure \ref{planar} for an illustration.
 Disc additions preserve $\di(v_1,v_3)$ as a diagonal, although
the underlying square may change.  Usually we add $\di(v_i,v_j)$ to the
current diagonal set.

 \medskip
 \noindent {\bf Operation 2: Handle addition}
 \medskip

 Let $f_{1}=v_{1}v_{2}v_{3}v_{4}$ and  $f_{2}=u_{1}u_{2}u_{3}u_{4}$ be
two squares of  $\Phi$. First, cut the open discs bounded by
$f_1$ and $f_2$ along their boundaries and remove them from the surface
$\Si$.
 Second, add a handle (cylinder) by identifying its two 
boundaries with the boundaries of $f_{1}$ and $f_{2}$ respectively.
Finally, add four new edges on the handle, each joining a vertex of
$f_1$ to a vertex of $f_2$, so that all resulting faces are squares. 
This operation is called a {\em handle addition.}
 The resulting embedding is also a quadrangular embedding.
  After applying a handle addition, the genus of the new
quadrangular embedding increases by one  if $\Si$ is orientable, and by
two if $\Si$ is nonorientable.

 We represent handle additions by the
planar diagrams shown in Figure~\ref{handle}.
 One of the two original squares is called the {\em outer square}
($f_1=v_1 v_2 v_3 v_4$ in Figure~\ref{handle}) and the other 
is called the {\em inner square} ($f_2=u_1 u_2 u_3 u_4$ in
Figure~\ref{handle}).
 The handle is the annular region between the outer and inner
squares.
 If the initial embedding is nonorientable, we may use the vertices of
the inner and outer squares in either order, as convenient, and the
resulting embedding is always nonorientable.

 If the initial embedding is orientable we must take more care.  Usually
we want the resulting embedding to also be orientable.
 When we add a handle to an orientable surface $\Si$ to create
a new orientable surface, a given direction around the handle
corresponds to the clockwise direction in $\Si$ at one end of the handle
and the counterclockwise direction in $\Si$ at the other end.
 In particular, consider the clockwise direction around a handle
as represented in the figure.  We assume this corresponds to the
clockwise direction in $\Si$ for the outer square; then it must
correspond to the counterclockwise direction in $\Si$ for the inner
square.  So if $v_1v_2v_3v_4$ and $u_1u_2u_3u_4$ are in clockwise order
in $\Si$, the figure has $v_1v_2v_3v_4$ and $u_1u_4u_3u_2$ in clockwise
order.
 For the new squares created by a handle addition, the clockwise order
of their vertices in the new surface is the clockwise order in the
figure.
 If we did use clockwise order $u_1u_2u_3u_4$ for the
inner square in the figure, we would add a twisted handle, and the new
embedding would be nonorientable.

 \begin{figure}[!hbtp]\refstepcounter{figure}\label{handle}
  \begin{center}
  \includegraphics[scale=0.9]{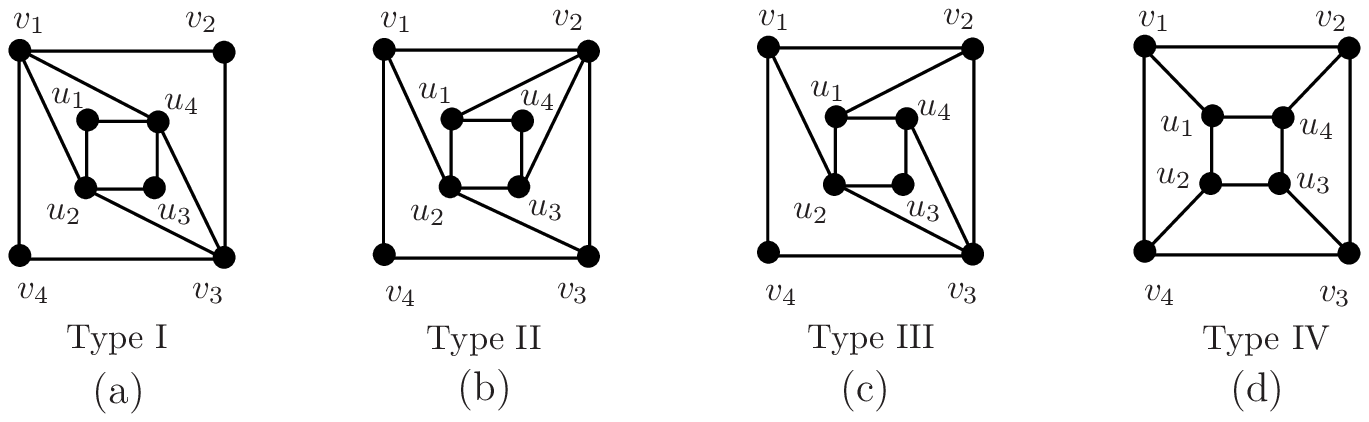} \\
  {Figure~\ref{handle}: The four types of handle  additions}
  \end{center}
  \end{figure}

  For two given squares, there are four different
types of handle additions between them based on the new edge
connections, which are listed in Figure~\ref{handle}.
 If we wish to be
specific, we use the labels in the figure and refer to a \emph{handle
addition of Type I, II, III or IV}, as appropriate.
 (Our Types I, II and IV correspond to Hartsfield's Types 1, 2
and 3, respectively.  Hartsfield did not use handle additions of Type
III.)
 Handle additions of Types I and III preserve $\di(v_1,v_3)$ and
$\di(u_2,u_4)$ as diagonals, although the underlying squares may change. 
Handle additions of Type II preserve $\di(v_1,v_3)$ and $\di(u_1,u_3)$.
Handle additions of Type IV are not guaranteed to preserve diagonals of
$f_1$ or $f_2$.

 \medbreak
 \noindent {\bf Operation 3: Crosscap addition}
 \medskip

  Let $f=v_1v_2v_3v_4$ be a square of  $\Phi$.
  Cut a disc from the interior of the square $f_1$, which leaves the surface $\Si$ with a hole.
 Then identify antipodal points of the boundary of the hole, which generates a crosscap inside of $f$.
 Finally, add two new edges $v_{1}v_{3} $ and
$v_{2}v_{4}$  passing through the new crosscap.  This
operation is called a \emph{crosscap addition}.
 See Figure \ref{crosscap}.
  The resulting embedding is a nonorientable
quadrangular embedding, and the Euler characteristic decreases by one
(so the genus increases by one if the original embedding was also
nonorientable). Neither diagonal of $f$ is a diagonal of either of the new
squares. (Hartsfield did not use crosscap additions.)

  \begin{figure}[!hbtp]\refstepcounter{figure}\label{crosscap}
  \begin{center}
  \includegraphics[scale=0.7]{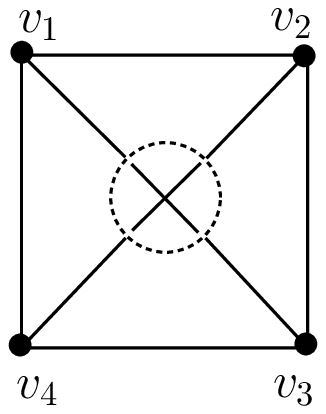} \\
  {Figure~\ref{crosscap}:  Crosscap addition}
  \end{center}
  \end{figure}


 \medskip
 \noindent {\bf Outline of the diagonal technique}
 \medskip

 We now provide a brief overview of the
diagonal technique, using the concepts and operations defined above.
 The idea is to replace squares in a known embedding by new squares
while adding edges and sometimes also vertices.
 In this paper the known embedding will be a quadrangulation, but in
general it may have some faces that are not squares.
 All or most of the vertices of the known embedding are partitioned into
a diagonal set of vertex pairs, so that the
vertices in each pair occur as diagonally opposite vertices
in an existing square.
 New vertices are first added in pairs using disc additions, adding
pairs to the diagonal set.
 Then most edges incident with the new vertices are added using handle
additions of Type I.  Each such addition usually takes two pairs of
vertices in the current diagonal set, uses an underlying square of one
pair as the outer square and an underlying square of the other pair as
the inner square, and adds four edges between the two pairs.
 
 A small number of edges may be added using handle additions of Types
II, III and IV and (in the nonorientable case) crosscap
additions.
  Often the exact details
(in particular, which underlying square is used for each diagonal pair)
of the handle additions of Type I do not matter, so they can be done in
a fairly arbitrary way, except that the necessary faces must be
constructed for any additions of handles of Types II, III and IV
and crosscaps.

 Hartsfield's diagonal technique is particularly suited for constructing
minimal quadrangulations, because it adds four (or sometimes two) edges
at a time, which is precisely what we need to get the embeddings
described in Theorem \ref{construction}.
 The `graphical surface' construction due to Craft \cite{CR}, which was
used in \cite{LA, LEYZ} to construct some minimal quadrangulations, can
be regarded as a special case of the diagonal technique.  It is
just the case where we start with an embedding of $C_4$ in the sphere
and use only disc additions and handle additions of Type I that preserve
the current diagonal set.

 Hartsfield's diagonal technique belongs to a more general class of
methods that construct embeddings, particularly orientable embeddings,
by adding handles, sometimes called \emph{tubes}, that carry specific
edges.  We mention a few examples of such methods.
 White \cite{WH70} and Pisanski \cite{PI1} (see also
\cite[Subsection 3.5.4]{GT}) used tubes to construct orientable
quadrangular embeddings of cartesian products of
bipartite graphs; their operations are
similar to our handle addition of Type IV.
 Lv and Chen \cite{LC19} used handle insertions where each handle
carries four or five edges to construct minimum orientable genus
embeddings of $K_{n,n,1}$ when $n$ is odd.
 Ma and Ren \cite{MR19} used tubes, typically added between triangular
faces and carrying five edges, to construct orientable minimum genus
embeddings of $C_m+K_n$ for $n \ge 5$ and $m \ge 6n-13$.

 \section{Proof of Theorem \ref{construction}}\label{constructionproof}

 In this section, we focus on the proof of Theorem \ref{construction}, 
which  is divided into two major cases --  orientable
surfaces and nonorientable surfaces.
 Apart from a few small cases, our proof is self-contained and
does not rely on earlier constructions of minimal quadrangulations.

 Denote the vertex set of a graph (or an embedding) of order $n$ by
$\{v_{1}, v_{2}, \dots,  v_{n}\}$.
 We know that in a disc addition or a handle addition of Type
I,
 if a diagonal of a diagonal set $\sD$ is used,
 then at least two new squares with the same diagonal are obtained. So this diagonal still occurs in the resulting embedding.
 We then replace the underlying square of the diagonal in
$\sD$ by an arbitrary choice of one of the two new squares
unless otherwise stated.
 For convenience,  
the resulting diagonal set is still denoted by $\sD$.
 Such situations occur frequently in the proof of Theorem
\ref{construction}.

 Except in some small cases, Theorem
\ref{construction} does not mention the surface in which an embedding
occurs.  The surface can always be determined from equation
(\ref{quadeq}), taking orientability into account.

 \subsection*{Orientable surfaces}

 Let $\Q(n,t)$ denote an orientable quadrangular embedding of a simple
graph with $n$ vertices and $\binom{n}2-t$ edges, for any integers $n
\ge 1$ and $t$ with $0 \le t \le \binom{n}2$.
 Let $T_n = \{t \;|\; 0 \le t \le n-4,\, t \equiv \frac12 n(n-5)$ (mod
$4$)$\}$ for each integer $n \ge 4$.
 The elements of $T_n$ form an arithmetic progression with difference
$4$.
 The condition $t \equiv \frac12 n(n-5)$ (mod $4$) is equivalent to $8
\,|\, n(n-5)-2t$ and so we must consider the value of $n \bmod 8$ to
determine $T_n$.
 We must show that there exist a $Q(4,2)$ in $S_0$ and $\Q(n, t)$ for
each $n \ge 4$ and $t \in T_n$.  We divide the
proof into the cases where $n$ is even and odd.

 Recall that when working with orientable surfaces we must pay close
attention to the clockwise order of vertices around each square.


 \begin{lemma}\label{orieven}
 There exist a $\Q(4,2)$ and $\Q(n,t)$ for each even $n \ge 4$ and $t
\in T_n$.
 \end{lemma}

 \begin{proof}
 Clearly,  a 4-cycle is a quadrangulation $\Q(4,2)$ of the sphere $S_{0}$.
 We have
 $T_4 = \{t \;|\; 0 \le t \le 0,\, t \equiv 2$ (mod $4$)$\} = \emptyset$
 and $T_6 = \{t \;|\; 0 \le t \le 2,\, t \equiv 3$ (mod
$4$)$\} = \emptyset$, so there is nothing else to prove for $n \le 6$.
 For $n \ge 8$ we proceed inductively.

 \begin{basis}
 There exist $\Q(8,t)$ for all $t \in T_8$.  In particular, there exists
a $\Q(8,0)$ with a perfect diagonal set.
 \end{basis}

 We have $T_8 = \{0, 4\}$. Hartsfield and Ringel \cite{HR1} gave a
quadrangular embedding $\Phi_{8}$ of $K_{8}$ in $S_{4}$, shown in Figure
\ref{4k8}.  This is the required $\Q(8,0)$, with a perfect diagonal
set using the squares shaded in the figure, namely
 $$\sD_8 =\{
  \diq(v_1,v_2,  v_1v_6v_2v_5),
  \diq(v_3,v_4, v_3v_7v_4v_8),
  \diq(v_5,v_6, v_5v_8v_6v_4),
  \diq(v_7,v_8, v_7v_1v_8v_2)
 \}.$$
 Also in \cite{HR1}, Hartsfield and Ringel constructed
a quadrangular embedding of the octahedral graph  $O_{8}$, which is the
required $\Q(8,4)$.

 \begin{figure}[!hbtp]\refstepcounter{figure}\label{4k8}
  \begin{center}
  \includegraphics[scale=0.85]{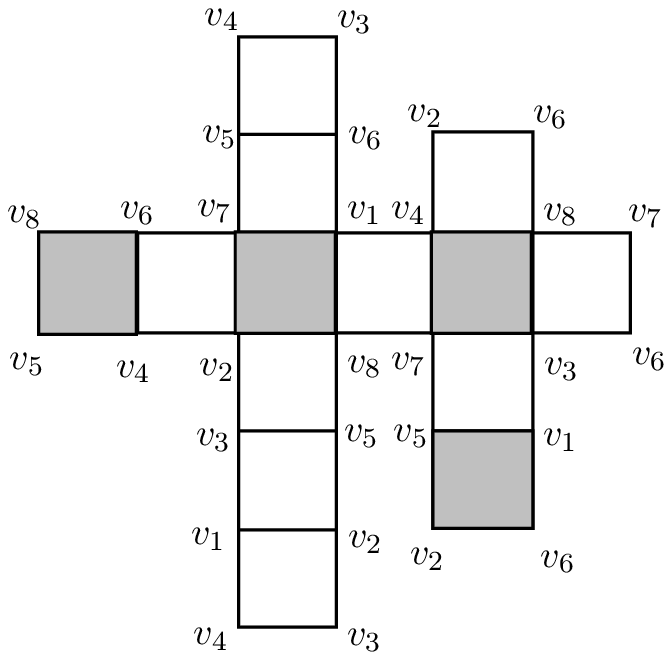} \\
  {Figure~\ref{4k8}}
  \end{center}
  \end{figure}

 \begin{inductionstep}
 Given a $\Q(n,0)$ with a perfect diagonal set, where $n=8k$, $k \ge 1$,
there exist $Q(p,t)$ for all $p \in \{n+2, n+4, n+6, n+8\}$ and $t \in
T_{p}$.  In particular, there exists a $\Q(n+8,0)$ with a perfect
diagonal set.
 \end{inductionstep}

 Suppose that $n=8k$, $k \ge 1$, and a $\Q(n,0)$, denoted by
$\Phi_n$, with a perfect diagonal set $\sD_n$ exists.
 Without loss of generality, we assume that
 $$ \sD_{n}=\{\di(v_{1},v_{2}),
	\di(v_{3},v_{4}), \dots,  \di(v_{n-1},v_{n}) \}.$$
 We construct the necessary embeddings in four stages.
 The first two handle additions of Type I in Stages 1 and 2, and the
last two handle additions of Type I in Stages 3 and 4, are setting up
squares needed for a Type IV handle addition in Stage 4.

 \begin{stage-oe}
 Suppose $p=n+2$.
 
 Since $p=n+2 \equiv 2$ (mod $8$) we have $p(p-5)
\equiv 2$ (mod $8$) and hence $\frac12 p(p-5) \equiv 1$ (mod $4$).
 Thus,
 $T_{n+2} = T_p =
 \{t \;|\; 0 \le t \le p-4,\, t \equiv 1$ (mod $4$)$\} =
 \{t \;|\; 0 \le t \le (n+2)-4,\, t \equiv 1$ (mod $4$)$\} =
 \{1, 5, \ldots, n-3\}$.
 \end{stage-oe}

 We start with $\Phi_n$.
 First employing a disc addition, we add the two vertices $v_{n+1}$ and
$v_{n+2}$ into the square with $\di(v_{1},v_{2})$ from $ \sD_{n}$ and
obtain the square $v_{n+1}v_{2}v_{n+2}v_{1}$ with a new diagonal
$\di(v_{n+1}, v_{n+2})$.
 Then, we apply $\frac{n}{2}-1$ successive handle additions of  Type I.
 During the process, the squares with $\di(v_{n+1}, v_{n+2})$ resulting
from previous additions are used as the outer squares and the squares
 with the diagonals from $\sD_{n}$ as  the inner squares.
 Moreover, the diagonals of $\sD_{n}$  are used in the order
$\di(v_{3},v_{4}), \di(v_{5},v_{6}), \dots, \di(v_{n-1}, v_{n})$.
 See Figure \ref{4kn2}.

 In our figures we often do not label vertices whose identity does not
matter, except that we use $x_1, x_2, \ldots$ to label vertices which
help to identify a square in later parts of the construction.
 Handle additions of Type I connect two diagonals using an
underlying square for each diagonal, unless otherwise specified.
 We shade squares that are \emph{reserved} for later use; these should
not be used as the inner or outer square in a handle addition until that
is explicitly specified.
 For example, the first two handle additions create reserved squares
$v_{n+1} v_{4} x_1 v_{3}$ and $v_{n+2} v_{5} x_2 v_{6}$, for use in
Stage 2 below.

 \begin{figure}[!hbtp]\refstepcounter{figure}\label{4kn2}
  \begin{center}
  \includegraphics[scale=0.8]{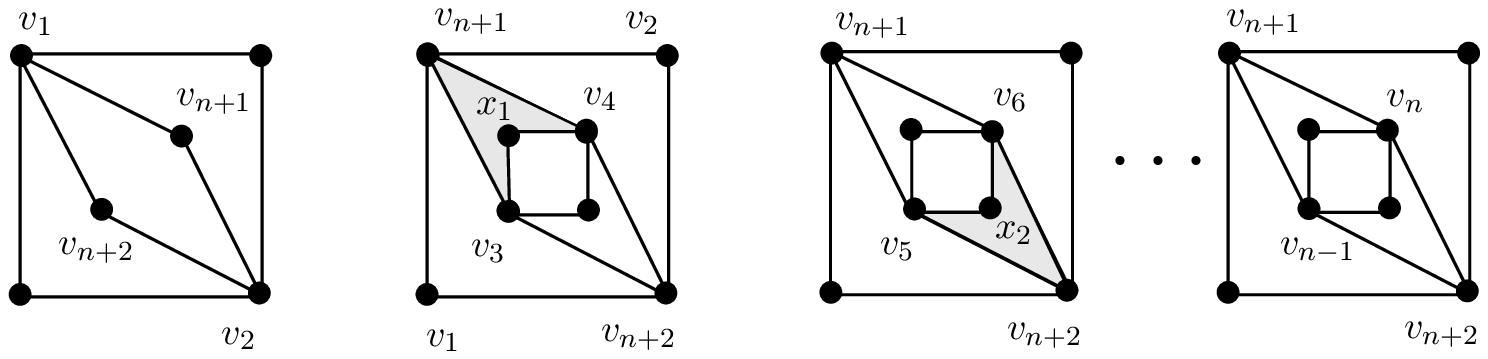} \\
  {Figure~\ref{4kn2}}
  \end{center}
  \end{figure}

 After the intial disc addition we have $\binom{n}2+4 =
\binom{n+2}2-(2n-3)$ edges.
 So this process constructs embeddings $\Q(n+2, t)$ for $t = 2n-3, 2n-7,
2n-11, \ldots, n-3, \ldots, 5, 1$, which includes all values $t \in
T_{n+2}$.
 Since the disc and handle additions join $v_{n+1}$ and $v_{n+2}$ to all of $v_1,
v_2, \dots, v_n$, but do not provide an edge $v_{n+1}v_{n+2}$, the final
result $\Phi_{n+2}$ is an embedding of $K_{n+2}-E_{1}$ where
$E_{1}=\{v_{n+1} v_{n+2}\}$.  It has a perfect diagonal set
 $$ \sD_{n+2}= \{ \di(v_1, v_2),
  \diq(v_3, v_4, v_{n+1} v_{4} x_1 v_{3}),
  \diq(v_5, v_6, v_{n+2} v_{5} x_2 v_{6}),
  \di(v_7, v_8),
  \di(v_9, v_{10}),
  \ldots, \di(v_{n+1}, v_{n+2}) \}. $$

 \begin{stage-oe}
 Suppose $p=n+4$.  Since $n+4 \equiv 4$ (mod $8$) we have
 $T_{n+4} = \{t \;|\; 0 \le t \le (n+4)-4,\, t \equiv 2$ (mod $4$)$\} =
\{2, 6, \ldots, n-2\}$.
 \end{stage-oe}

 Similarly to Stage 1, starting from $\Phi_{n+2}$ with $\sD_{n+2}$ we
can employ a disc addition to add vertices $v_{n+3}$ and $v_{n+4}$ and
obtain a diagonal $\di(v_{n+3}, v_{n+4})$, then
employ $\frac{n}{2}$ handle additions of Type I.
 The first two handle additions use the reserved squares
from Stage 1 as inner squares.  They create new reserved squares with
new diagonals, for use in Stage 3 below.
 See Figure \ref{4kn4}.
 After the initial disc addition we have $\binom{n+2}2+3 =
\binom{n+4}2-(2n+2)$ edges.
 So this process creates embeddings $\Q(n+4,t)$ for $t = 2n+2, 2n-2,
2n-6, \ldots, n-2, \ldots, 6, 2$, which includes all $t \in T_{n+4}$.

 \begin{figure}[!hbtp]\refstepcounter{figure}\label{4kn4}
  \begin{center}
  \includegraphics[scale=0.8]{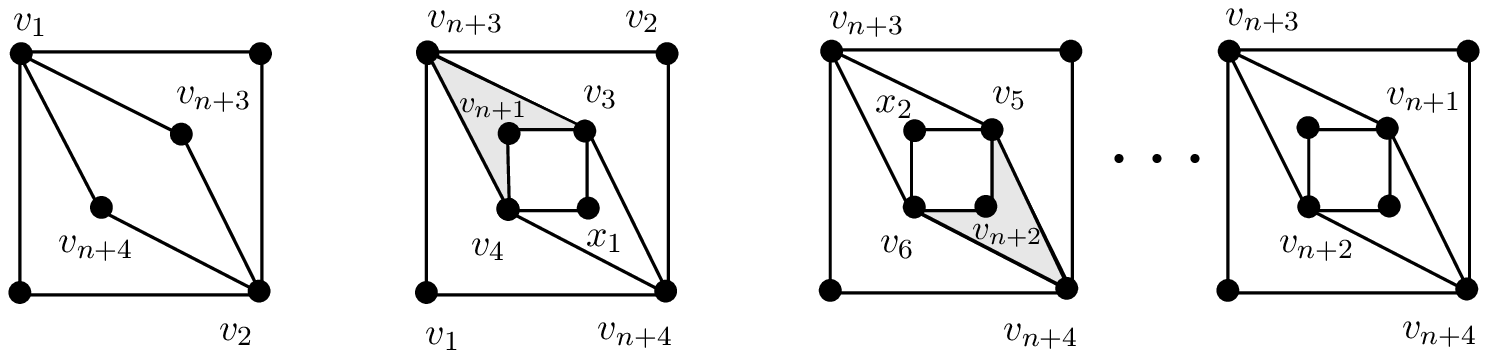} \\
  {Figure~\ref{4kn4}}
  \end{center}
  \end{figure}

 The final result $\Phi_{n+4}$ is a quadrangular embedding of
$K_{n+4}-E_{2}$ where $E_{2}=\{v_{n+1} v_{n+2},\ab  v_{n+3} v_{n+4}
\}$.  It has a perfect diagonal set
 \begin{align*}
 \sD_{n+4} &= \{ \di(v_1, v_2),
 \di(v_3, v_4), \di(v_5, v_6), \ldots, \di(v_{n-1}, v_n), \\
 & \qquad\qquad \diq(v_{n+1}, v_{n+3}, v_{n+1} v_4 v_{n+3} v_3),
 \diq(v_{n+2}, v_{n+4}, v_{n+2} v_5 v_{n+4} v_6)
 \}
 \end{align*}
 where $v_{n+1} v_4 v_{n+3} v_3$ and $v_{n+2} v_5 v_{n+4} v_6$
are the two reserved squares.

 \begin{stage-oe} Suppose $p=n+6$.
 Since $n+6 \equiv 6$ (mod $8$) we have $T_{n+6} = \{t \;|\; 0 \le t \le
(n+6)-4,\, t \equiv 3$ (mod $4$)$\} = \{3, 7, \ldots, n-1\}$.
 \end{stage-oe}

 Similarly to Stages 1 and 2, from $\Phi_{n+4}$ with $\sD_{n+4}$ we can
employ a disc addition to add vertices $v_{n+5}$ and $v_{n+6}$, and then
$\frac n2 +1$ handle additions of Type I.
 The last two handle additions create reserved squares for use in Stage 4
below.
 See Figure \ref{4kn6}.
 After the initial disc addition we have $\binom{n+4}2+2 =
\binom{n+6}2-(2n+7)$ edges.
 So this process creates embeddings $\Q(n+6,t)$ for $t = 2n+7, 2n+3,
2n-1, \ldots, n-1, \ldots, 7, 3$, which includes all $t \in T_{n+6}$.

 \begin{figure}[!hbtp]\refstepcounter{figure}\label{4kn6}
  \begin{center}
  \includegraphics[scale=0.8]{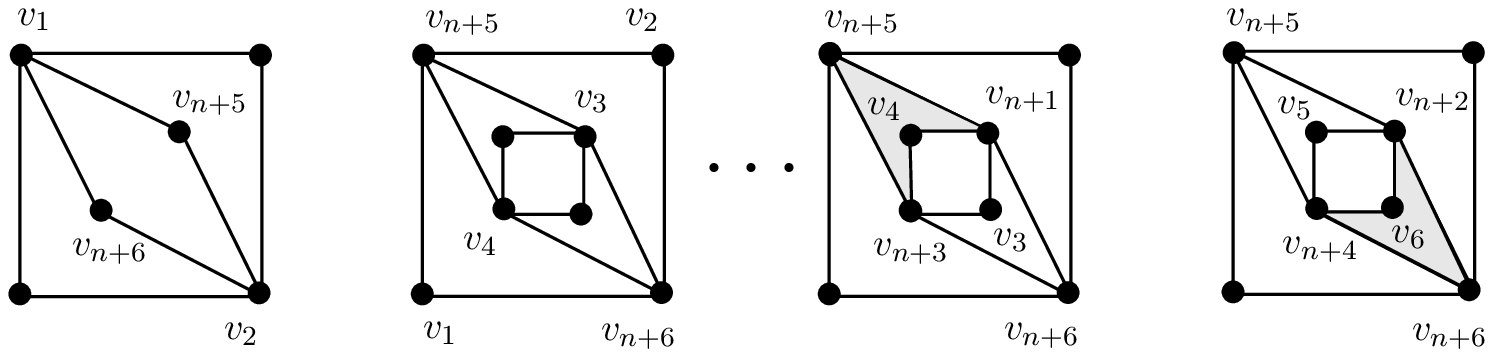} \\
  {Figure~\ref{4kn6}}
  \end{center}
  \end{figure}

 The final result $\Phi_{n+6}$ is a quadrangular embedding of 
$K_{n+6}-E_{3}$ where
 $E_{3}=\{ v_{n+1} v_{n+2},\ab v_{n+3} v_{n+4},\ab v_{n+5} v_{n+6}\}$.
 It has a perfect diagonal set
 \begin{align*}
 \sD_{n+6} &= \{
    \di(v_1,v_2), \di(v_3, v_4), \ldots, \di(v_{n-1}, v_n), \\
    &\qquad\qquad \diq(v_{n+1}, v_{n+3}, v_{n+1} v_{4} v_{n+3} v_{n+5}),
	\diq(v_{n+2}, v_{n+4}, v_{n+2} v_{n+6} v_{n+4} v_{6}),
	\di(v_{n+5}, v_{n+6}) \}
 \end{align*}
 where $v_{n+1} v_{4} v_{n+3} v_{n+5}$ and $v_{n+2} v_{n+6}
v_{n+4} v_{6}$ are the two reserved squares.

 \begin{stage-oe} Suppose $p=n+8$.
 Since $n+8 \equiv 0$ (mod $8$) we have $T_{n+8} = \{t \;|\; 0 \le t \le
(n+8)-4,\, t \equiv 0$ (mod $4$)$\} = \{0, 4, \ldots, n+4\}$.
 \end{stage-oe}

 Similarly to the previous stages, from $\Phi_{n+6}$ with $\sD_{n+6}$ we can
employ a disc addition to add vertices $v_{n+7}$ and $v_{n+8}$, and then
$\frac n2 +2$ handle additions of Type I.
 The last two handle additions of Type I create reserved squares
 $v_{n+1}v_{n+7}v_{n+3}v_{n+5}$ and
 $v_{n+2}v_{n+6}v_{n+4}v_{n+8}$, which we then use for a Type IV handle
addition.
 See Figure \ref{4kn8}.
 After the initial disc addition we have $\binom{n+6}2+1 =
\binom{n+8}2-(2n+12)$ edges.
 So this process creates embeddings $\Q(n+6,t)$ for $t = 2n+12, 2n+8,
2n+4, \ldots, n+4, \ldots, 4, 0$, which
includes all $t \in T_{n+8}$.

 \begin{figure}[!hbtp]\refstepcounter{figure}\label{4kn8}
  \begin{center}
  \includegraphics[scale=0.9]{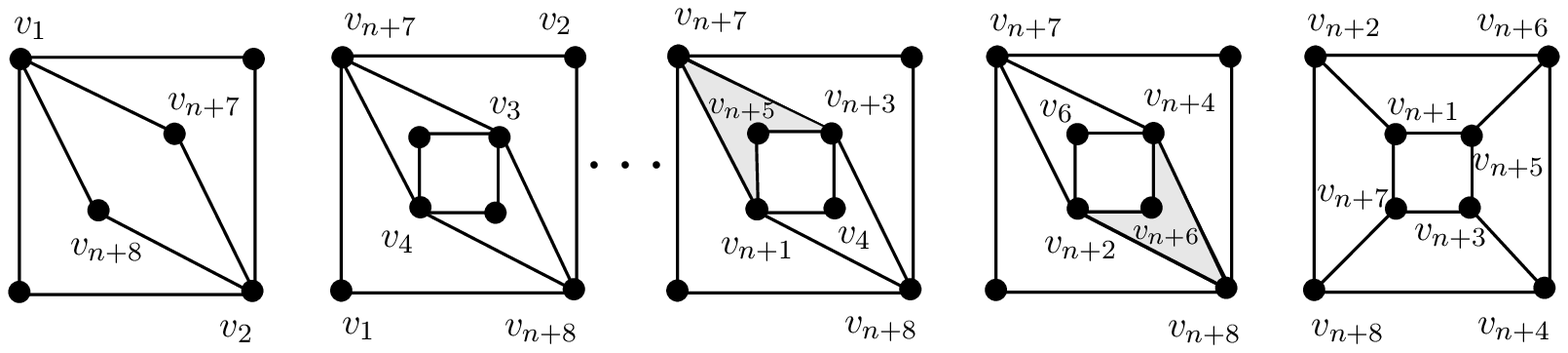} \\
  {Figure~\ref{4kn8}}
  \end{center}
  \end{figure}

 The final result $\Phi_{n+8}$ is a quadrangular embedding of $K_{n+8}$.
It is a $\Q(n+8,0)$ with perfect diagonal set
 \begin{align*}
 \sD_{n+8} &= \{ \di(v_1, v_2), \di(v_3, v_4), \ldots,
	\di(v_{n-1}, v_n), \\
	&\qquad\qquad \di(v_{n+1}, v_{n+3}), 
	\di(v_{n+2}, v_{n+4}), \di(v_{n+5}, v_{n+6}), 
	\di(v_{n+7}, v_{n+8})\}.
 \end{align*}

 This completes the proof of the induction step.
 Now the small cases ($n=4$ and $6$), the basis, and the induction step
together imply Lemma \ref{orieven}.
 \end{proof}


 \begin{lemma}\label{oriodd}
 There exists a $\Q(n,t)$ for each odd $n \ge 5$ and $t \in T_n$.
 \end{lemma}

 \begin{proof}
 We proceed inductively.

 \begin{basis}
 There exist $\Q(5,t)$ for all $t \in T_5$.  In particular, there exists
a $\Q(5,0)$ with a full diagonal set $\sD_5$, which contains
$v_1$-nearly-perfect and $v_2$-nearly-perfect subsets.
 \end{basis}

 We have $T_5 = \{0\}$ so we only need to find the specified embedding
$\Q(5,0)$.
 We use the embedding $\Phi_{5}$ of $K_5$ in $S_{1}$ with five squares, shown
in Figure \ref{k5}.  It has a full diagonal set
 \begin{align*}
  \sD_5 &= \{
	\diq(v_1, v_5, v_1 v_4 v_5 v_2),
	\diq(v_3, v_4, v_3 v_5 v_4 v_2),
	\diq(v_4, v_5, v_4 v_1 v_5 v_3),
	\diq(v_2, v_3, v_2 v_4 v_3 v_1)
 \}
 \end{align*}
 where the first two elements form a $v_2$-nearly-perfect subset and the
last two elements form a $v_1$-nearly-perfect subset.

 \begin{figure}[!hbtp]\refstepcounter{figure}\label{k5}
  \begin{center}
  \includegraphics[scale=0.8]{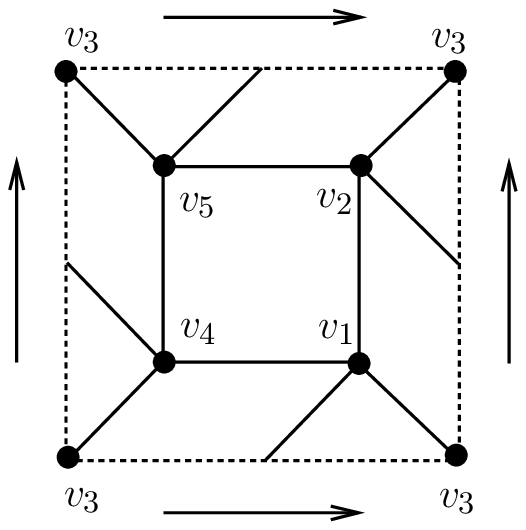} \\
  {Figure~\ref{k5} }
  \end{center}
  \end{figure}

 \begin{inductionstep}
 Suppose $n=8k+5$, $k \ge 0$, and we are given a $\Q(n,0)$ with a full
diagonal set having $v_1$-nearly-perfect and $v_2$-nearly-perfect
subsets.  Then  there exist $Q(p,t)$ for all $p \in \{n+2, n+4, n+6,
n+8\}$ and $t \in T_p$.  In particular, there exists a $\Q(n+8,0)$ with 
a full diagonal set having $v_1$-nearly-perfect and $v_2$-nearly-perfect
subsets.
 \end{inductionstep}

 Suppose that $n=8k+5$, $k \ge 0$, and there is a $\Q(n,0)$, denoted by
$\Phi_n$, with a full diagonal set $\sD_n$, as described.  Write the
$v_2$-nearly-perfect and $v_1$-nearly-perfect subsets as
 \begin{align*}
 %
  \sD'_n &= \{ \di(y_1, y_2), \di(y_3, y_4), \dots,
	\di(y_{n-4}, y_{n-3}), \di(y_{n-2}, v_1) \} \hbox{ and}\\
  \sD''_n &= \{ \di(z_1, z_2), \di(z_3, z_4), \dots, 
	\di(z_{n-4}, z_{n-3}), \di(z_{n-2}, v_2) \},
 \end{align*}
 respectively.
 Thus, $\{y_1, y_2, \dots, y_{n-2}\} = \{z_1, z_2, \dots, z_{n-2}\} =
\{v_3, v_4, \dots, v_n\}$.
 We construct the necessary embeddings in four stages.
 Note that the last few handle additions in each stage help to set up
squares needed for the handle additions of Type II and III in later
stages.


 \begin{stage-oo}
 Suppose $p=n+2$.  Since $n+2 \equiv 7$ (mod $8$) we have
 $T_{n+2} = \{t \;|\; 0 \le t \le (n+2)-4,\, t \equiv 3$ (mod $4$)$\}
	= \{3, 7, \dots, n-2\}$.
 \end{stage-oo}

  We start with $\Phi_n$.  First, by employing a disc addition, we add
two vertices $v_{n+1}$ and $v_{n+2}$ into the square with $\di(y_1,y_2)$
from $\sD_{n}$ and obtain the square $v_{n+1}y_2v_{n+2}y_1$ with
$\di(v_{n+1}, v_{n+2})$.
 Then, we apply $\frac{n-3}{2}$ successive handle additions of Type I.
 During the process, the squares with $\di(v_{n+1}, v_{n+2})$ resulting
from previous additions are used as the outer squares and the squares
 with the diagonals from $\sD'_{n}$, in the order $\di(y_3,y_4),
\di(y_5,y_6), \dots, \di(y_{n-2},v_1)$, as the inner squares.
 The final handle addition creates a square $v_{n+1} x_1
v_{n+2} v_1$ that is reserved for Stage 2 below.
 See Figure \ref{3kn2}.

 \begin{figure}[!hbtp]\refstepcounter{figure}\label{3kn2}
  \begin{center}
  \includegraphics[scale=0.9]{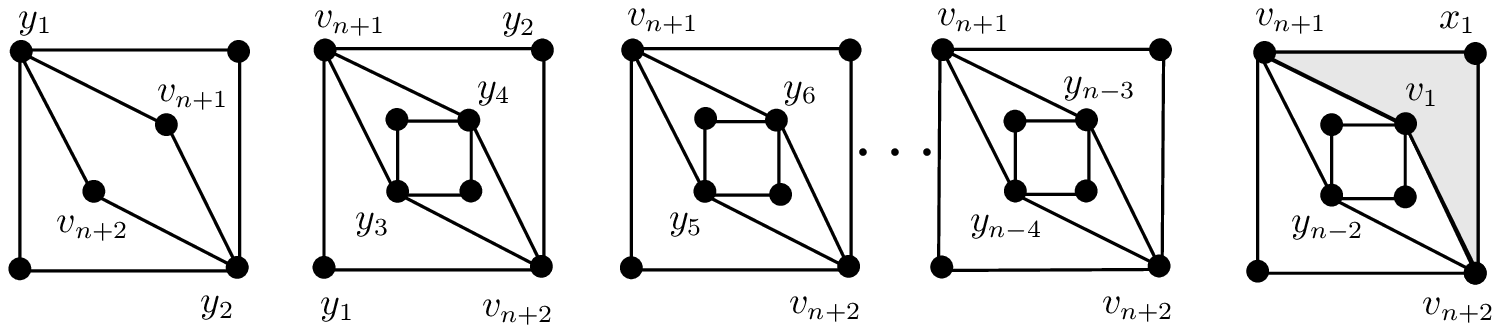} \\
  {Figure~\ref{3kn2} }
  \end{center}
  \end{figure}

 After the intial disc addition we have $\binom{n}2+4 =
\binom{n+2}2-(2n-3)$ edges.
 So this process constructs embeddings $\Q(n+2, t)$ for $t = 2n-3, 2n-7,
2n-11, \ldots, n-2, \ldots, 7, 3$, which includes all values $t \in
T_{n+2}$.
 Since the disc and handle additions join $v_{n+1}$ and
$v_{n+2}$ to all of $v_3, v_4, v_5, \dots, v_n$ and to $v_1$, but not to
$v_2$, and do not provide an edge $v_{n+1} v_{n+2}$, the final result
$\Phi_{n+2}$ is an embedding of $K_{n+2}-E_{3}$ where $E_{3}=\{v_2
v_{n+1}, v_2 v_{n+2}, v_{n+1} v_{n+2}\}$.  It has a full diagonal set 
$\sD_{n+2}= \sD_n \cup \{\di(v_{n+1}, v_{n+2})\}$ with
$v_2$-nearly-perfect and $v_1$-nearly perfect subsets
 $$\sD'_{n+2}= \sD'_n \cup \{\di(v_{n+1}, v_{n+2})\}
	\hbox{\quad and\quad}
   \sD''_{n+2}= \sD''_n \cup \{\di(v_{n+1}, v_{n+2})\},$$
 respectively.

 \begin{stage-oo}
 Suppose $p=n+4$.  Since $n+4 \equiv 1$ (mod $8$) we have
 $T_{n+4} = \{t \;|\; 0 \le t \le (n+4)-4,\, t \equiv 2$ (mod $4$)$\} =
\{2, 6, \ldots, n-3\}$.
 \end{stage-oo} 

 Starting from $\Phi_{n+2}$ with $\sD''_{n+2}$ we can employ a disc
addition to add vertices $v_{n+3}$ and $v_{n+4}$, then $\frac{n-1}{2}$
handle additions of Type I.  The second last of these creates a reserved
square, which is used in a final handle addition of Type II, along with
the reserved square from Stage 1.  This creates two reserved squares which
provide new diagonals with specific underlying squares, for use in Stage
3 below.
 See Figure \ref{3kn4}.
 After the initial disc addition we have $\binom{n+2}2+1 =
\binom{n+4}2-(2n+4)$ edges.
 So this process creates embeddings $\Q(n+4,t)$ for $t = 2n+4, 2n,
2n-4, \ldots, n-3, \ldots, 6, 2$, which includes all $t \in T_{n+4}$.

 \begin{figure}[!hbtp]\refstepcounter{figure}\label{3kn4}
  \begin{center}
  \includegraphics[scale=0.9]{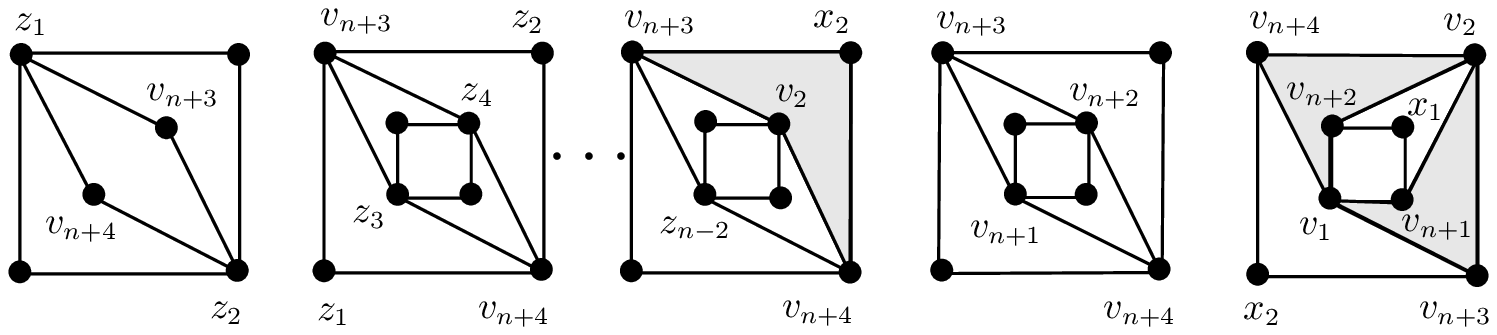} \\
  {Figure~\ref{3kn4}}
  \end{center}
  \end{figure}

 The final result $\Phi_{n+4}$ is a quadrangular embedding of
$K_{n+4}-E_{2}$ where $E_{2}=\{v_{n+1} v_{n+2},\ab v_{n+3} v_{n+4}
\}$.
 If we define a diagonal set using the last two reserved squares, i.e.,
 $$\sD^+_{n+4} = \{\diq(v_{n+1}, v_{n+3}, v_{n+1} v_2 v_{n+3} v_1),
	\diq(v_{n+2}, v_{n+4}, v_{n+2} v_1 v_{n+4} v_2)\}$$
 then $\Phi_{n+4}$ has a full diagonal set $\sD_{n+4} = \sD_n \cup
\sD^+_{n+4}$ with $v_2$-nearly-perfect and $v_1$-nearly-perfect subsets
$\sD'_{n+4} = \sD'_n \cup \sD^+_{n+4}$ and
$\sD''_{n+4} = \sD''_n \cup \sD^+_{n+4}$, respectively.

 \begin{stage-oo}
 Suppose $p=n+6$.  Since $n+6 \equiv 3$ (mod $8$) we have
 $T_{n+6} = \{t \;|\; 0 \le t \le
(n+6)-4,\, t \equiv 1$ (mod $4$)$\} = \{1, 5, \ldots, n\}$.
 \end{stage-oo} 

 Starting from $\Phi_{n+4}$ with $\sD''_{n+4}$, we can employ a disc
addition to add vertices $v_{n+5}$ and $v_{n+6}$, then $\frac{n+1}2$
handle additions of Type I, then a handle addition of Type III.
 See Figure \ref{3kn6}.
 The last four handle additions create and use up a number of reserved
squares; the net effect is that the two reserved squares from the Type
II addition in Stage 2 are used up, and two new reserved squares 
 $v_{n+1} v_{n+5} v_{n+3} v_1$ and $v_{n+2} v_{n+6} v_{n+4} v_2$ are
created for use in Stage 4 below.
 After the initial disc addition we have $\binom{n+4}2+2 =
\binom{n+6}2-(2n+7)$ edges.
 So this process creates embeddings $\Q(n+6,t)$ for $t=2n+7, 2n+3,
\dots, n, \dots, 5, 1$, which includes all $t \in T_{n+6}$.

 \begin{figure}[!hbtp]\refstepcounter{figure}\label{3kn6}
  \begin{center}
  \includegraphics[scale=0.9]{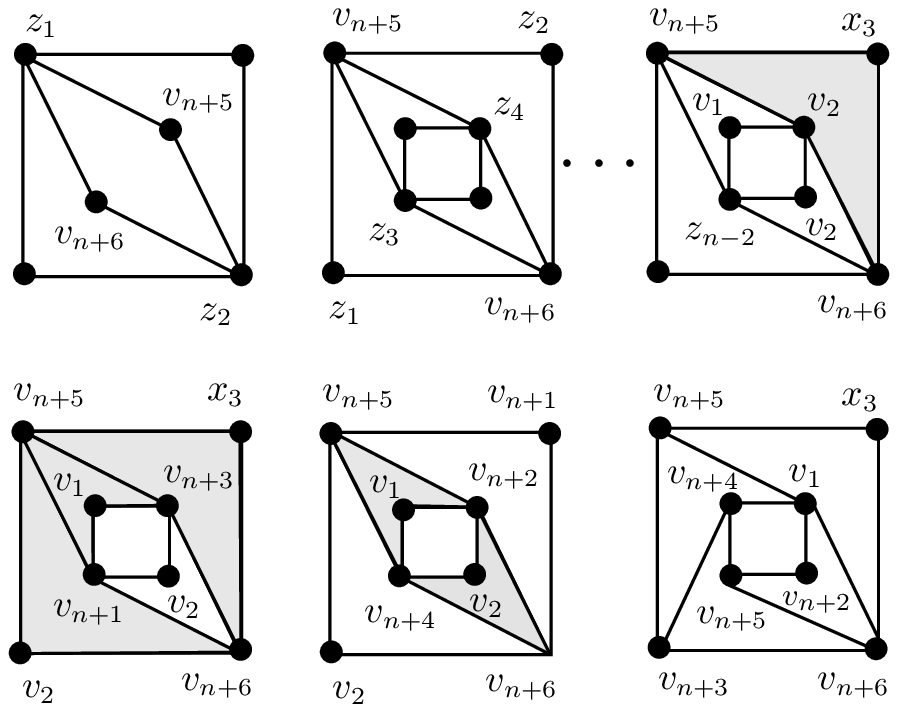} \\
  {Figure~\ref{3kn6}}
  \end{center}
  \end{figure}

 The final result $\Phi_{n+6}$ is a quadrangular embedding of
$K_{n+6}-E_{1}$ where $E_{1}=\{v_{n+1} v_{n+2}\}$.
 If we define a diagonal set containing the two unused reserved squares,
 $$\sD^+_{n+6} = \{
	\diq(v_{n+1}, v_{n+3}, v_{n+1} v_{n+5} v_{n+3} v_1),
	\diq(v_{n+2}, v_{n+4}, v_{n+2} v_{n+6} v_{n+4} v_2),
 	\di(v_{n+5}, v_{n+6})
	\}$$
 then $\Phi_{n+6}$ has a full diagonal set $\sD_{n+6} = \sD_n \cup
\sD^+_{n+6}$ with $v_2$-nearly-perfect and $v_1$-nearly-perfect subsets
$\sD'_{n+6} = \sD'_n \cup \sD^+_{n+6}$ and
$\sD''_{n+6} = \sD''_n \cup \sD^+_{n+6}$, respectively.

 \begin{stage-oo}
 Suppose $p=n+8$.  Since $n+8 \equiv 5$ (mod $8$) we have
 $T_{n+8} = \{t \;|\; 0 \le t \le
(n+8)-4,\, t \equiv 0$ (mod $4$)$\} = \{0, 4, \ldots, n+3\}$.
 \end{stage-oo} 

 Starting from $\Phi_{n+6}$ with $\sD''_{n+6}$, we can employ a disc
addition to add vertices $v_{n+7}$ and $v_{n+8}$.  Then we use
$\frac{n+3}2$ handle additions of Type I using diagonals in the order
$\di(z_1,z_2)$, $\di(z_3, z_4)$, $\dots$, $\di(z_{n-2}, v_2)$,
$\di(v_{n+5}, v_{n+6})$ and lastly, using the two reserved squares from
Stage 3, $\di(v_{n+1}, v_{n+3})$ and $\di(v_{n+2}, v_{n+4})$.
 We finish with a handle addition of Type III.
 See Figure \ref{3kn8}.
 The last three handle additions create and use up three additional
reserved squares.
 After the initial disc addition we have $\binom{n+6}2+3 =
\binom{n+8}2-(2n+10)$ edges.
 So this process creates embeddings $\Q(n+8,t)$ for $t=2n+10, 2n+6,
\dots, n+3, \dots, 4, 0$, which includes all $t \in T_{n+8}$.

 \begin{figure}[!hbtp]\refstepcounter{figure}\label{3kn8}
  \begin{center}
  \includegraphics[scale=0.9]{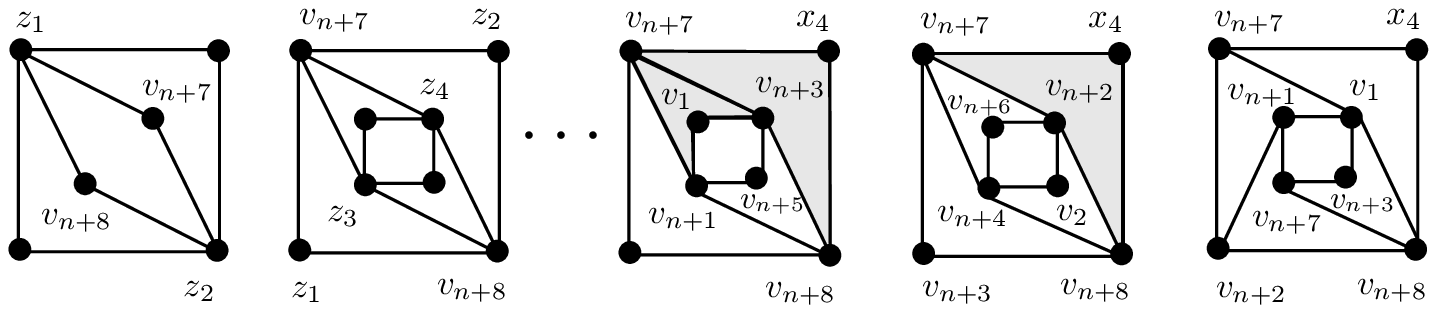} \\
  {Figure~\ref{3kn8}}
  \end{center}
  \end{figure}

 The final result $\Phi_{n+8}$ is a quadrangular embedding of
$K_{n+8}$.
 If we define a diagonal set
 $$\sD^+_{n+8} = \{
	\di(v_{n+1}, v_{n+3}), \di(v_{n+2}, v_{n+4}),
 	\di(v_{n+5}, v_{n+6}), \di(v_{n+7}, v_{n+8})
 \}$$
 then $\Phi_{n+8}$ is a $\Q(n+8,0)$ with a full diagonal set $\sD_{n+8}
= \sD_n \cup \sD^+_{n+8}$ having $v_2$-nearly-perfect and
$v_1$-nearly-perfect subsets $\sD'_n \cup \sD^+_{n+8}$ and $\sD''_n \cup
\sD^+_{n+8}$, respectively.

 This completes the proof of the induction step.
 Now the basis and the induction step together imply Lemma \ref{oriodd}.
 \end{proof}

 \subsection*{Nonorientable surfaces}

 Let $\nQ(n,t)$ denote a nonorientable quadrangular embedding of a simple
graph with $n$ vertices and $\binom{n}2-t$ edges, for any integers $n
\ge 1$ and $t$ with $0 \le t \le \binom{n}2$.
 Let $\nT_n = \{t \;|\; 0 \le t \le n-4,\, t \equiv \frac12 n(n-5)$ (mod
$2$)$\}$ for each integer $n \ge 4$.
 The elements of $\nT_n$ form an arithmetic progression with difference
$2$.
 The condition $t \equiv \frac12 n(n-5)$ (mod $2$) is equivalent to $4
\,|\, n(n-5)-2t$ and so we must consider the value of $n \bmod 4$ to
determine $\nT_n$.
 We must show that 
there exist $\nQ(n, t)$ for each $n \ge 4$ and
$t \in \nT_n$ except when $(n,t)=(5,0)$, and that there exists a
$\nQ(6,3)$ in $N_2$.
 We divide the proof into the cases where $n \le 6$, where $n \ge 8$ is
even, and where $n \ge 7$ is odd.

 When working with nonorientable surfaces we may use a square $v_i v_j
v_k v_\ell$ in its reverse order $v_i v_\ell v_k v_j$ whenever
convenient.

 \begin{lemma}\label{nonorismall}
 Suppose that $4 \le n \le 6$.  Then there exist $\nQ(n,t)$ for each $t
\in \nT_n$, except when $(n,t)=(5,0)$.  There also exists a $\nQ(6,3)$
in $N_2$.
 \end{lemma}

 \begin{proof}
 We have $\nT_4=\{0\}$.  The complete graph $K_{4}$ admits a
quadrangular embedding with three squares (every $4$-cycle
in $K_4$ bounds a face) in the projective plane $N_1$, which is a
$\nQ(4,0)$.

 We have $\nT_5 = \{0\}$.  If $(n,t) = (5,0)$ then the graph is $K_5$,
but by Theorem \ref{complete} there is no nonorientable quadrangular
embedding of $K_5$, so no embedding exists for this case.

 \begin{figure}[!hbtp]\refstepcounter{figure}\label{k60}
  \begin{center}
  \includegraphics[scale=0.9]{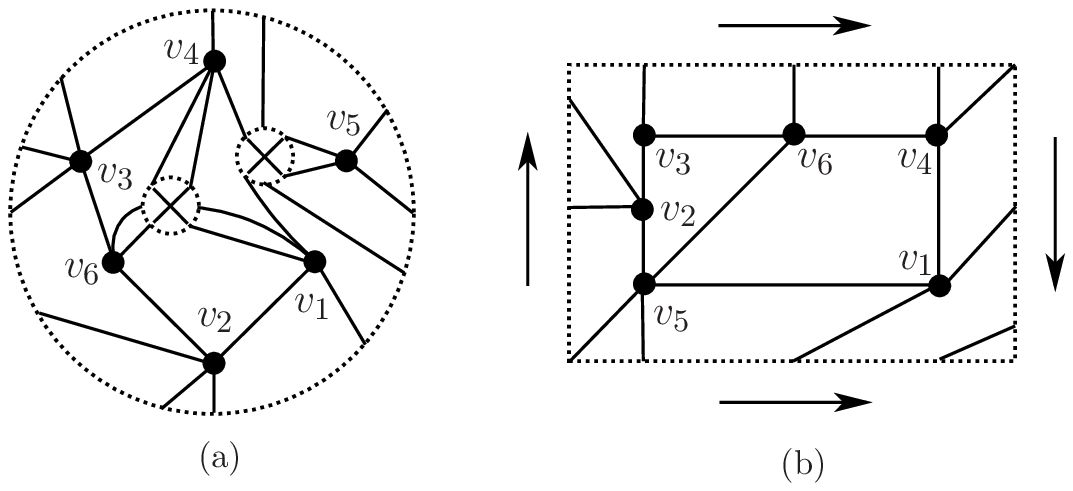} \\
  {Figure~\ref{k60}}
  \end{center}
  \end{figure}

 We have $\nT_6=\{1\}$.
Figure \ref{k60}(a) shows  that $K_{6}-E_{1}$ has a
quadrangular embedding in $N_{3}$ with $E_{1}=\{ v_{5}v_{6}\}$, which is
a $\nQ(6,1)$.
 Also, Figure \ref{k60}(b) shows that $K_{6}-E_{3}$ has a quadrangular
embedding $\Phi^{-}_{6}$ in $N_{2}$ with $E_{3}=\{v_{1}v_{3}, 
v_{2}v_{6},  v_{3}v_{4} \}$, which is a $\nQ(6,3)$ in $N_{2}$.
 \end{proof}


 \begin{lemma}\label{nonorieven}
 There exists a $\nQ(n,t)$ for each even $n \ge 8$ and $t \in \nT_n$.
 \end{lemma}

 \begin{proof}
 We proceed inductively.

 \begin{basis}
 There exist $\nQ(8,t)$ for all $t \in \nT_8$.
 In particular, there exists a $\nQ(8,2)$ that embeds a graph $K_8 -
\{u_1 u_2, u_3 u_4\}$ and has a perfect diagonal set $\sD_8$ in which
$u_1$, $u_2$, $u_3$ and $u_4$ belong to four distinct diagonals.
 \end{basis}

 \begin{figure}[!hbtp]\refstepcounter{figure}\label{nk8}
  \begin{center}
  \includegraphics[scale=0.8]{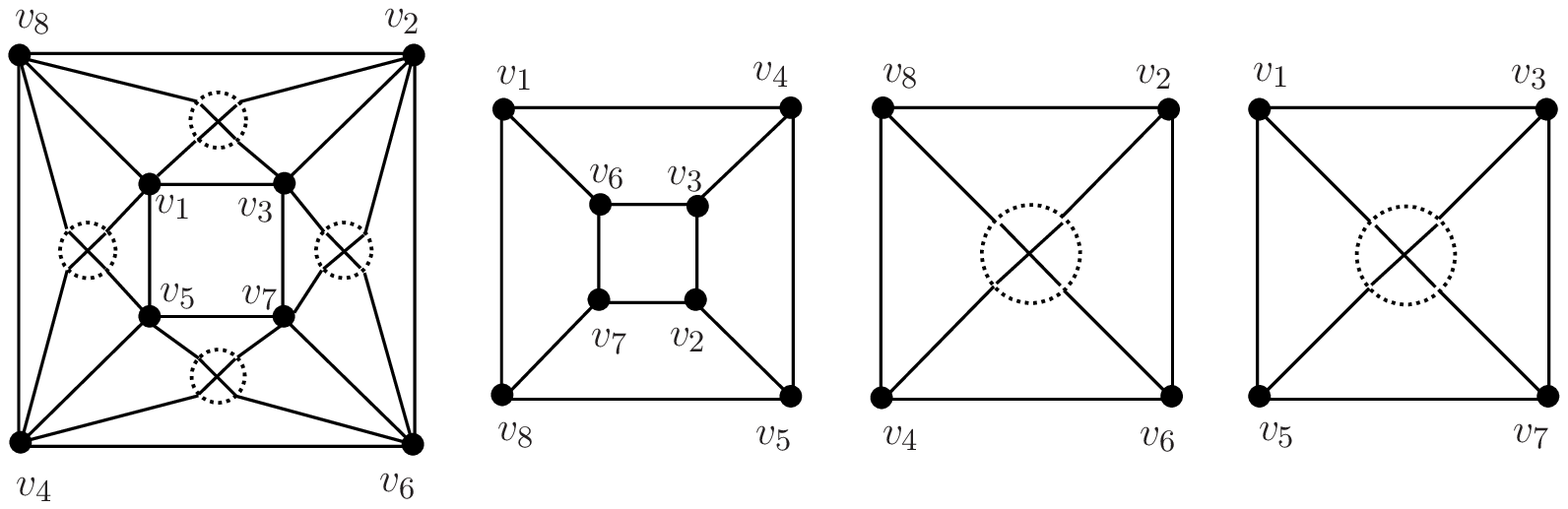} \\
  {Figure~\ref{nk8}}
  \end{center}
  \end{figure}

 We have $\nT_8 = \{0, 2, 4\}$.  At left in Figure \ref{nk8} is a
$\nQ(8,8)$ obtained by four crosscap additions from the usual spherical
(planar) embedding of a cube, to which we apply a handle addition of
Type IV followed by two crosscap additions.  The result is a
quadrangular embedding of $K_8$ in $N_8$, and along the way we construct
embeddings $\nQ(8,t)$ with $t=8, 4, 2, 0$, which includes all $t \in
\nT_8$.

 We examine the $\nQ(8,2)$ obtained by performing the handle addition
and the first crosscap addition, but not the second crosscap addition. 
 This is a quadrangular embedding $\Phi^-_8$ of $K_8 - \{v_1v_7,
v_3v_5\}$ in $N_7$.
 There is a perfect diagonal set
 $$\sD_8 = \{\diq(v_8, v_1, v_8 v_4 v_1 v_5),
	\diq(v_2, v_3, v_2 v_8 v_3 v_1),
	\diq(v_4, v_5, v_4 v_6 v_5 v_7),
	\diq(v_6, v_7, v_6 v_2 v_7 v_3)
 \}.$$
 Taking $u_1=v_1$, $u_2=v_7$, $u_3=v_3$ and $u_4=v_5$, we see that the
conditions for the particular $\nQ(8,2)$ are satisfied.

 \begin{inductionstep}
 Suppose $n = 4k$, $k \ge 2$, and we are given
 a $\nQ(n,2)$ that embeds a graph $K_n - \{u_1 u_2, u_3 u_4\}$ and has a
perfect diagonal set $\sD_n$ in which $u_1$, $u_2$, $u_3$ and $u_4$
belong to four distinct diagonals.
 Then there exist $\nQ(p,t)$ for all $p \in \{n+2, n+4\}$ and $t \in
\nT_p$.
 In particular, there exists
 a $\nQ(n+4,2)$ that embeds a graph $K_{n+4} - \{u'_1 u'_2, u'_3 u'_4\}$
and has a perfect diagonal set $\sD_{n+4}$ in which $u'_1$, $u'_2$,
$u'_3$ and $u'_4$ belong to four distinct diagonals.
 \end{inductionstep}

 Suppose that $n=4k$, $k \ge 2$, and there is a $\nQ(n,2)$, denoted by
$\Phi^-_n$, as described above.  We may assume without loss of
generality that
 $$\sD_n = \{\di(v_1, v_2), \di(v_3, v_4), \di(v_5, v_6), \di(v_7, v_8),
\dots, \di(v_{n-1}, v_n)\}$$
 and that $u_1=v_1$, $u_2=v_3$, $u_3=v_5$ and $u_4=v_7$.  Thus,
$\Phi^-_n$ is an embedding of $K_n - E_2$ with $E_2 = \{v_1 v_3, v_5
v_7\}$.  We construct the necessary embeddings in two stages.

 \begin{stage-ne}
 Suppose $p = n+2$.  Since $n+2 \equiv 2$ (mod $4$) we have
 $\nT_{n+2} = \{t \;|\; 0 \le t \le (n+2)-4,\, t \equiv 1$ (mod $2$)$\}
= \{1, 3, 5 \dots, n-3\}$.
 \end{stage-ne}

 We start with $\Phi^-_n$.  First, by applying a disc addition, we add
two vertices $v_{n+1}$ and $v_{n+2}$ into the square with diagonal
$\di(v_1, v_2)$ from $\sD_n$.  Then we emply a handle addition of Type I
with $\di(v_3, v_4)$, creating two reserved squares, one of which, $v_3
v_{n+2} v_4 x_1$ is for use in Stage 2 below.  The other reserved
square, $v_{n+1} v_3 v_{n+2} v_1$, we use
immediately for a crosscap addition.  We then perform $\frac n2 - 2$
handle additions of Type I using diagonals $\di(v_5, v_6)$, $\di(v_7,
v_8)$, $\dots$, $\di(v_{n-1}, v_n)$ in that order.  The first of these
creates another reserved square $v_5 v_{n+1} v_6 x_2$ for use in Stage
2.
 See Figure \ref{nekn2}.

 \begin{figure}[!hbtp]\refstepcounter{figure}\label{nekn2}
  \begin{center}
  \includegraphics[scale=0.7]{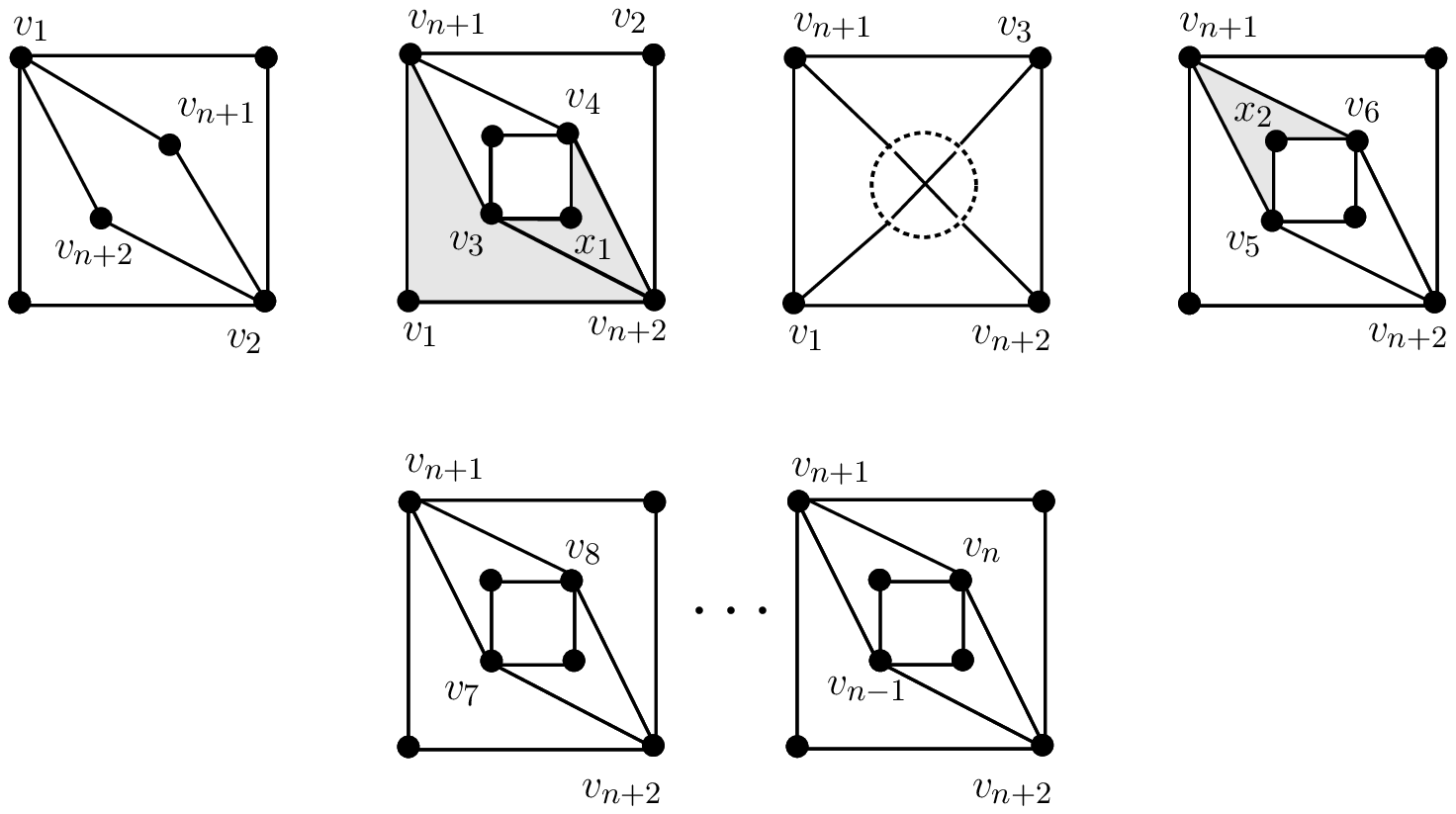} \\
  {Figure~\ref{nekn2}}
  \end{center}
  \end{figure}

 After the initial disc addition we have $\binom{n}2 + 2 = \binom{n+2}2
- (2n-1)$ edges.  Since handle additions add four edges and crosscap
additions add two edges, our process
creates embeddings $\nQ(n+2, t)$ for $t = 2n-1, 2n-5, 2n-7, 2n-11,
2n-15, \dots, 5, 1$.
 However, we do not use the squares created by the crosscap addition in
any later handle addition, so we can just omit the crosscap addition. 
This creates embeddings $\nQ(n+2, t)$ for $t=2n-1, 2n-5, 2n-9, \dots, 7,
3$.
 Combining these, we have $\nQ(n+2, t)$ for $t = 2n-1, 2n-5, 2n-7, 2n-9,
\dots, n-3, \dots, 3, 1$, i.e., for all odd $t$ with $1 \le t \le 2n-1$
except $t = 2n-3$.  Since $n \ge 8$, this includes all $t \in
\nT_{n+2}$.

 The final result $\Phi_{n+2}$ is an embedding of $K_{n+2}-E_1$ where
$E_1 = \{v_5 v_7\}$.  Using the reserved squares containing $x_1$ and
$x_2$, we see that $\Phi_{n+2}$ has a perfect diagonal set
 \begin{align*}
 \sD_{n+2} &= \{
	\di(v_1, v_2),
	\diq(v_3, v_4, v_3 v_{n+2} v_4 x_1),
	\diq(v_5, v_6, v_5 v_{n+1} v_6 x_2),
	\\ & \qquad	
	\diq(v_7, v_8), \di(v_9, v_{10}), \dots, \di(v_{n+1}, v_{n+2})
 \}.
 \end{align*}

 \begin{stage-ne}
 Suppose $p = n+4$.  Since $n+4 \equiv 0$ (mod $4$) we have
 $\nT_{n+4} = \{t \;|\; 0 \le t \le
(n+4)-4,\, t \equiv 0$ (mod $2$)$\} = \{0, 2, 4, \dots, n\}$.
 \end{stage-ne}

 Starting from $\Phi_{n+2}$ with $\sD_{n+2}$ we can employ a disc
addition to add vertices $v_{n+3}$ and $v_{n+4}$, using the reserved
underlying square for $\di(v_5, v_6)$ as the outer square.  This creates
a reserved square $v_{n+1} v_6 v_{n+3} v_5$ which we will use to create
a new diagonal later.
 Then we perform a handle addition of Type I using $\di(v_7, v_8)$,
creating a reserved square $v_{n+3} v_7 v_{n+4} v_5$, which we use
immediately for a crosscap addition.  Next we employ a handle addition
of Type I, using the reserved underlying square for $\di(v_3, v_4)$ as
the inner square.  This creates a new reserved square $v_{n+2} v_4
v_{n+4} v_3$, which we will also use to create a new diagonal later.
 Finally we perform $\frac n2 - 2$ handle additions of Type I using all
remaining diagonals from $\sD_{n+2}$.
 See Figure \ref{nekn4}.

 \begin{figure}[!hbtp]\refstepcounter{figure}\label{nekn4}
  \begin{center}
  \includegraphics[scale=0.7]{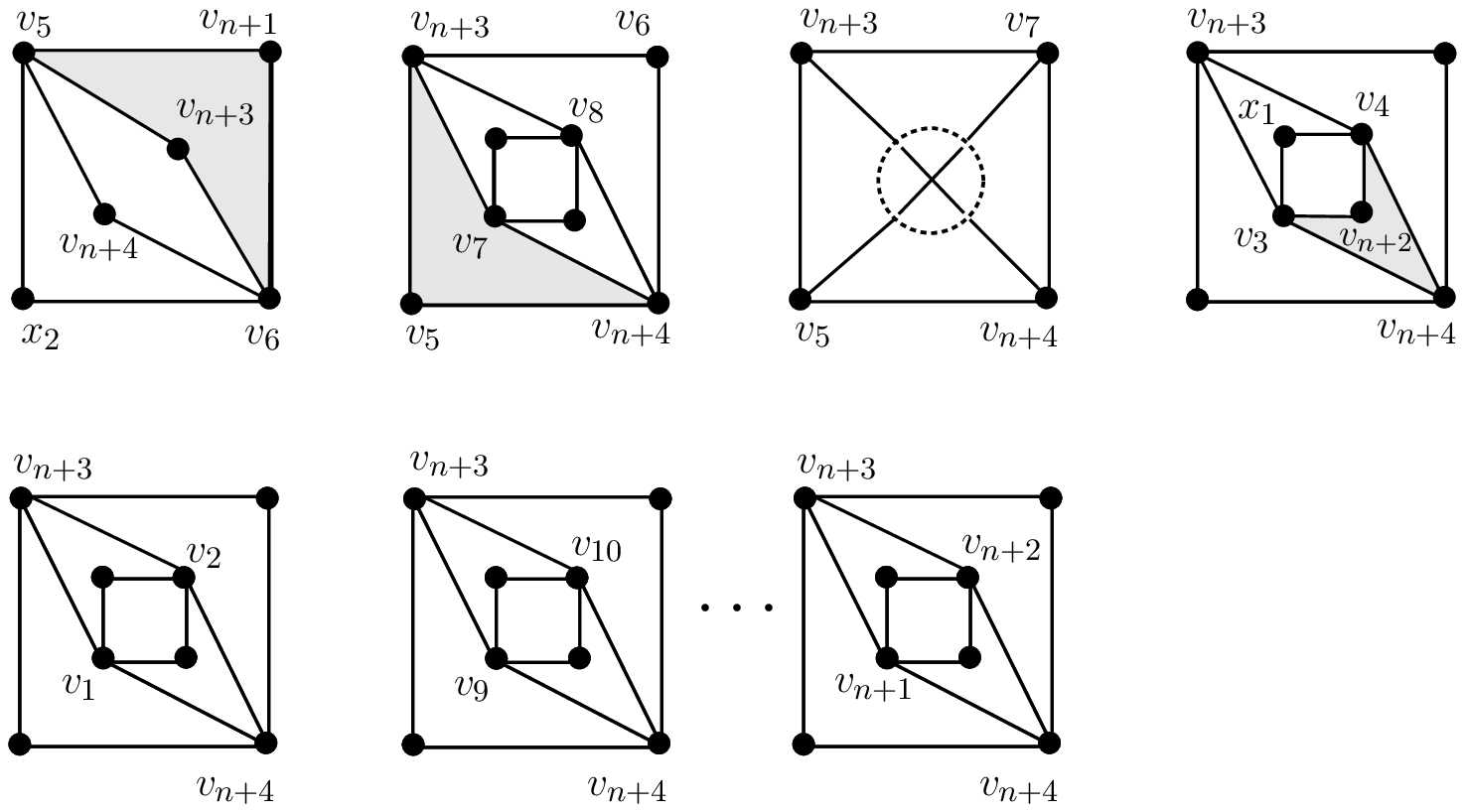} \\
  {Figure~\ref{nekn4}}
  \end{center}
  \end{figure}

 After the initial disc addition we have $\binom{n+2}2 + 3 =
\binom{n+4}2 - (2n+2)$ edges.  Since handle additions add four edges and
crosscap additions add two edges, our process creates embeddings
$\nQ(n+4, t)$ for $t = 2n+2, 2n-2, 2n-4, 2n-8, 2n-12, \dots, 4, 0$.
 However, we do not use the square created by the crosscap addition in
any later handle addition, so we can just omit the crosscap addition. 
This produces $\nQ(n+4, t)$ for $t = 2n+2, 2n-2, 2n-6, 2n-10, \dots, 6, 2$.
 Combining these, we have $\nQ(n+4, t)$ for $t = 2n+2, 2n-2, 2n-4, 2n-6,
2n-8, \dots, n, \dots, 2, 0$, i.e., for all even $t$ with $0 \le t \le
2n+2$ except $t=2n$.  Since $n \ge 8$, this includes all $t \in
\nT_{n+4}$.

 If we perform all handle additions but omit the crosscap addition we
obtain an embedding $\Phi^-_{n+4}$ of $K_{n+4}-E_2$ where $E_2 = \{v_5
v_7, v_{n+3} v_{n+4}\}$.  This is a $\nQ(n+4,
2)$.  Using two reserved squares to create new diagonals, we see that
$\Phi^-_{n+4}$ has a perfect diagonal set
 \begin{align*}
 \sD_{n+4} &= \{
	\di(v_1, v_2), \di(v_3, v_4), \di(v_5, v_6), \di(v_7, v_8),
		\dots, \di(v_{n-1}, v_n),
	\\ & \qquad	
	\diq(v_{n+1}, v_{n+3}, v_{n+1} v_6 v_{n+3} v_5),
	\diq(v_{n+2}, v_{n+4}, v_{n+2} v_4 v_{n+4} v_3)
 \}.
 \end{align*}
 Taking $u'_1 = v_5$, $u'_2 = v_7$, $u'_3 = v_{n+3}$, $u'_4 = v_{n+4}$,
we see that the required properties for a particular $\nQ(n+4, 2)$ hold.

 This completes the proof of the induction step.  Now the basis and the
induction step together imply Lemma \ref{nonorieven}.
 \end{proof}

 \begin{lemma}\label{nonoriodd}
 There exists a $\nQ(n,t)$ for each odd $n \ge 7$ and $t \in \nT_n$.
 \end{lemma}

 \begin{proof}
 We proceed inductively.  Our argument requires a slightly technical
induction hypothesis, so we make the following definition.

 \begin{definition}
 A $\nQ(n,3)$, where $n \equiv 3$ (mod $4$), is said to have
\emph{Property P} if the following conditions (a), (b) and (c) hold.

(a) The graph is $K_n-E_3$ where $E_3 = \{v_1 v_{n-1}, v_1 v_n, v_{n-3}
v_{n-2}\}$.

(b) There is a full diagonal set
 $$
  \sD_n = \{
	\di(v_1, v_3), \di(v_2, v_3),
	\di(y_1, y_2), \di(y_3, y_4), \dots, \di(y_{n-8}, y_{n-7}),
	\di(v_{n-3}, v_{n-1}), \di(v_{n-2}, v_n)
 \}
 $$
 where $\{y_1, y_2, \dots, y_{n-7}\} = \{v_4, v_5, \dots, v_{n-4}\}$.


(c) There is a square $v_2 v_{n-1} x_1 v_n$ (the exact identity of $x_1$
does not matter) that is not an underlying square for $\sD_n$.  (We
reserve this square for later use.)
 \end{definition}

 \begin{basis}
 There exist $\nQ(7,t)$ for all $t \in \nT_7$.  In particular, there exists
a $\nQ(7,3)$ with Property P.
 \end{basis}

 We have $\nT_7=\{1,3\}$.  Figure \ref{k7} shows a quadrangular
embedding $\Phi_7$ of $K_7-E_1$ in $N_5$ with $E_1 = \{v_1v_6\}$.  This
is the required $\nQ(7,1)$.
 If we delete the two edges $v_1 v_7$ and $v_4 v_5$ of $\Phi_7$, we
create a face, bounded by a $4$-cycle and containing a crosscap, which
we can remove and replace by a disc (this is the inverse of a crosscap
addition).
 We obtain a quadrangular embedding $\Phi^-_7$ of $K_7-E_3$ in $N_4$,
which is a $\nQ(7,3)$.  We verify that $\Phi^-_7$ also has Property P.
 (a)~The missing edges form $E_3 = \{v_1 v_6, v_1 v_7, v_4 v_5\}$, as
required.
 (b)~There is a full diagonal set
 \begin{align*}
  \sD_7 = \{
	\diq(v_1, v_3, v_1 v_2 v_3 v_5),
	\diq(v_2, v_3, v_2 v_1 v_3 v_5),
	\diq(v_4, v_6, v_4 v_2 v_6 v_3),
	\diq(v_5, v_7, v_5 v_2 v_7 v_6)
 \}
 \end{align*}
 of the required form.
 %
 %
 (c)~There is a square $v_2 v_6 v_5 v_7$ that is not an underlying square for
$\sD_7$, as required.
 Thus, $\Phi^-_7$ has Property P.

 \begin{figure}[!hbtp]\refstepcounter{figure}\label{k7}
  \begin{center}
  \includegraphics[scale=0.75]{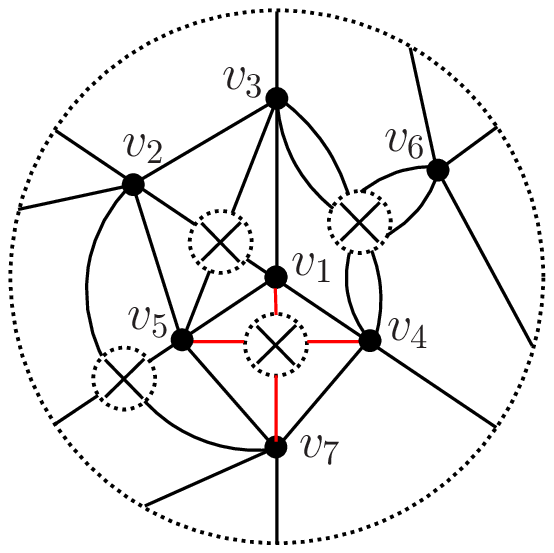} \\
  {Figure~\ref{k7}}
  \end{center}
  \end{figure}

 \begin{inductionstep}
 Given a $\nQ(n,3)$ with Property P, where $n=4k+3$, $k \ge 1$,
there exist $\nQ(p,t)$ for all $p \in \{n+2, n+4\}$ and $t \in
\nT_{p}$.  In particular, there exists a $\nQ(n+4,3)$ with Property P.
 \end{inductionstep}

 Suppose that $n=4k+3$, $k \ge 1$, and there is a $\nQ(n,3)$, denoted by
$\Phi^-_n$, satisfying Property P.  We construct the necessary
embeddings in two stages.

 \begin{stage-no}
 Suppose $p = n+2$.  Since $n+2 \equiv 1$ (mod $4$) we have
 $\nT_{n+2} = \{t \;|\; 0 \le t \le (n+2)-4,\, t \equiv 0$ (mod $2$)$\}
= \{0, 2, 4, \dots, n-3\}$.
 \end{stage-no}

 We start with $\Phi^-_n$.  First, by applying a disc addition, we add
two vertices $v_{n+1}$ and $v_{n+2}$ into the square with $\di(v_{n-3},
v_{n-1})$ from $\sD_n$, creating a reserved square $v_{n-3}
v_{n+1} v_{n-1} x_2$ for use in Stage 2 below.  Then we employ a handle addition of
Type I with $\di(v_{n-2}, v_n)$, creating two reserved squares, one of
which, $v_{n-2} v_{n+2} v_n x_3$, is for use in Stage 2.  The other
reserved square, $v_{n+1} v_{n-3} v_{n+2} v_{n-2}$, we use immediately
for a crosscap addition.
 We then perform $\frac{n-5}2$ handle additions of Type I using
diagonals $\di(y_1, y_2)$, $\di(y_3, y_4)$, $\dots$, $\di(y_{n-8},
y_{n-7})$, $\di(v_1, v_3)$, in that order.
 The last handle addition creates a reserved square $v_{n+1} v_1 v_{n+2}
x_4$, which we then use, along with the reserved square $v_2 v_{n-1} x_1
v_n$ from Property P(c), in a handle addition of Type II, which creates
a further reserved square $ v_{n+1} v_2 v_{n+2} x_4$ for use in
Stage 2.
 See Figure \ref{nokn2}.

 \begin{figure}[!hbtp]\refstepcounter{figure}\label{nokn2}
  \begin{center}
  \includegraphics[scale=0.7]{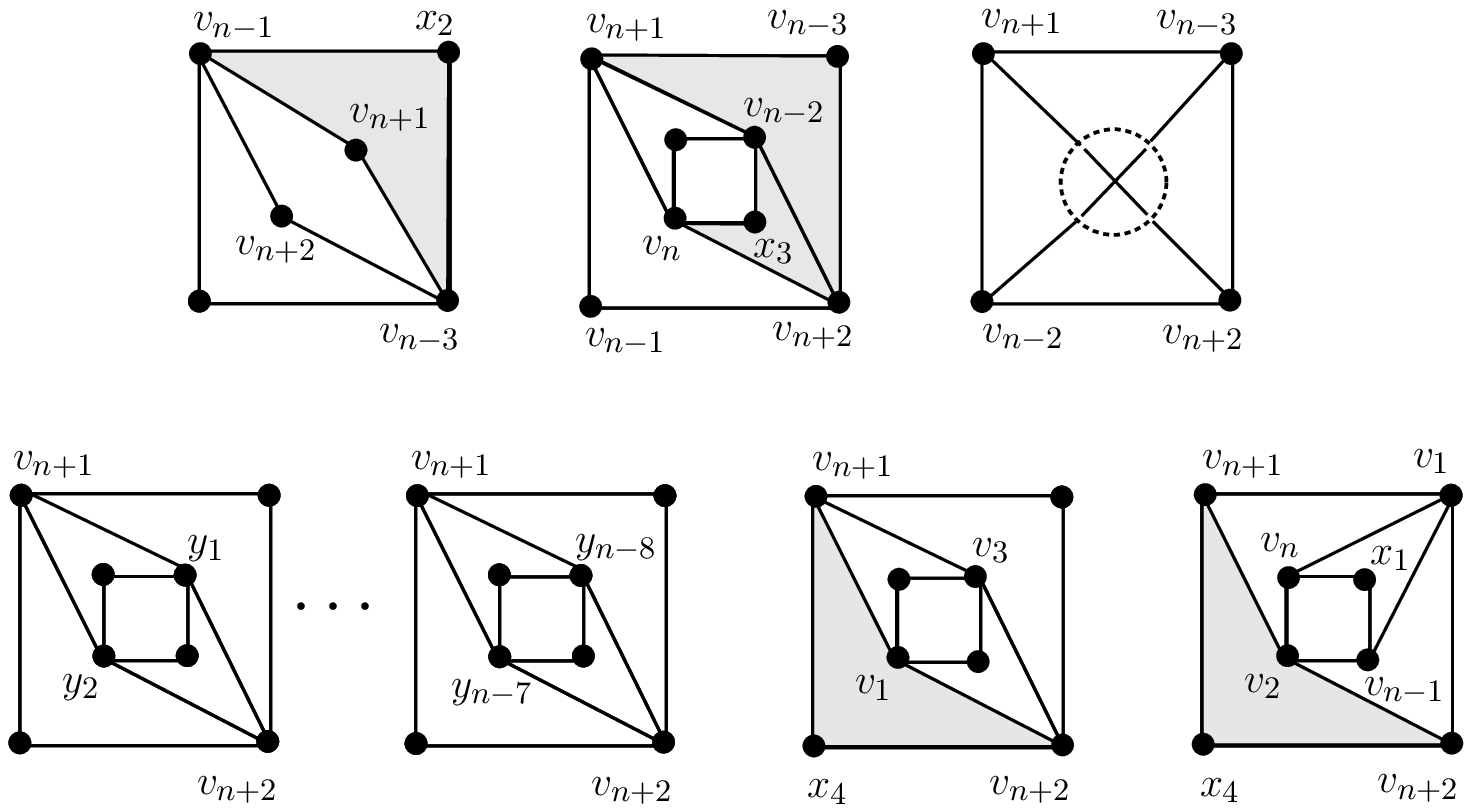} \\
  {Figure~\ref{nokn2}}
  \end{center}
  \end{figure}

 After the initial disc addition we have $\binom{n}2+1 = \binom{n+2}2 -
2n$ edges.  Since handle additions add four edges and crosscap additions
add two edges, our process creates embeddings $\nQ(n+2,t)$ for $t = 2n,
2n-4, 2n-6, 2n-10, 2n-14, \dots, 4, 0$.
 However, we do not use the squares created by the crosscap
addition in any later handle addition, so we can just omit the crosscap
addition.  This creates embeddings $\nQ(n+2, t)$ for $t=2n, 2n-4, 2n-8,
\dots, 6, 2$.
 Combining these, we have $\nQ(n+2,t)$ for $t = 2n, 2n-4, 2n-6, 2n-8,
\dots, n-3, \dots, 2, 0$, i.e., for all even $t$ with $0 \le t \le 2n$
except $t = 2n-2$.  Since $n \ge 7$, this includes all $t \in
\nT_{n+2}$.

 If we perform all handle additions but omit the crosscap addition, we
obtain an embedding $\Phi^-_{n+2}$ of $K_{n+2}-E_2$ where $E_2 =
\{v_{n-2} v_{n-3}, v_{n+1} v_{n+2}\}$ (the edges of the omitted crosscap
addition).
 Using the reserved squares containing $x_2$, $x_3$ and $x_4$, we see
that $\Phi^-_{n+2}$ has a full diagonal set
 \begin{align*}
  \sD_{n+2} &= \{
	\di(v_1, v_3), \di(v_2, v_3),
	\di(y_1, y_2), \di(y_3, y_4), \dots, \di(y_{n-8}, y_{n-7}),
	\diq(v_{n-3}, v_{n-1}, v_{n-3} v_{n+1} v_{n-1} x_2),
	\\ &\qquad
	\diq(v_{n-2}, v_n, v_{n-2} v_{n+2} v_n x_3),
	\diq(v_{n+1}, v_{n+2}, v_{n+1} v_2 v_{n+2} x_4)
 \}.
 \end{align*}

 \begin{stage-no}
 Suppose $p = n+4$.  Since $n+4 \equiv 3$ (mod $4$) we have
 $\nT_{n+4} = \{t \;|\; 0 \le t \le (n+4)-4,\, t \equiv 1$ (mod $2$)$\}
= \{1, 3, 5, \dots, n\}$.
 \end{stage-no}

 Starting from $\Phi^-_{n+2}$ with $\sD_{n+2}$, we can employ a disc
addition to add vertices $v_{n+3}$ and $v_{n+4}$, using the reserved
underlying square for $\di(v_{n-3}, v_{n-1})$ as the outer square.
 This creates a reserved square $v_{n+1} v_{n-1} v_{n+3} v_{n-3}$ which
we will use to satisfy Property P(b).
 Then we perform a handle addition of Type I, using the reserved
underlying square for $\di(v_{n-2}, v_n)$ as the inner square.  This
creates two reserved squares, one of which, $v_{n+2} v_n v_{n+4}
v_{n-2}$, we will use to satisfy Property P(b).  The other, $v_{n+3}
v_{n-2} v_{n+4} v_{n-3}$, we use immediately for a crosscap addition.
 Next we employ a handle addition of Type I, using the reserved
underlying square for $\di(v_{n+1}, v_{n+2})$ as the inner square.  This
creates a reserved square $v_{n+2} v_2 v_{n+1} v_{n+4}$, which we use
immediately for a crosscap addition.
 We then perform $\frac{n-5}2$ handle additions of Type I using
diagonals $\di(y_1, y_2)$, $\di(y_3, y_4)$, $\dots$, $\di(y_{n-8},
y_{n-7})$, $\di(v_1, v_3)$, in that order.
 The last handle addition creates a reserved square $v_1 v_{n+4} x_5
v_{n+3}$, which we will use to satisfy Property P(c).
 See Figure \ref{nokn4}.

 \begin{figure}[!hbtp]\refstepcounter{figure}\label{nokn4}
  \begin{center}
  \includegraphics[scale=0.7]{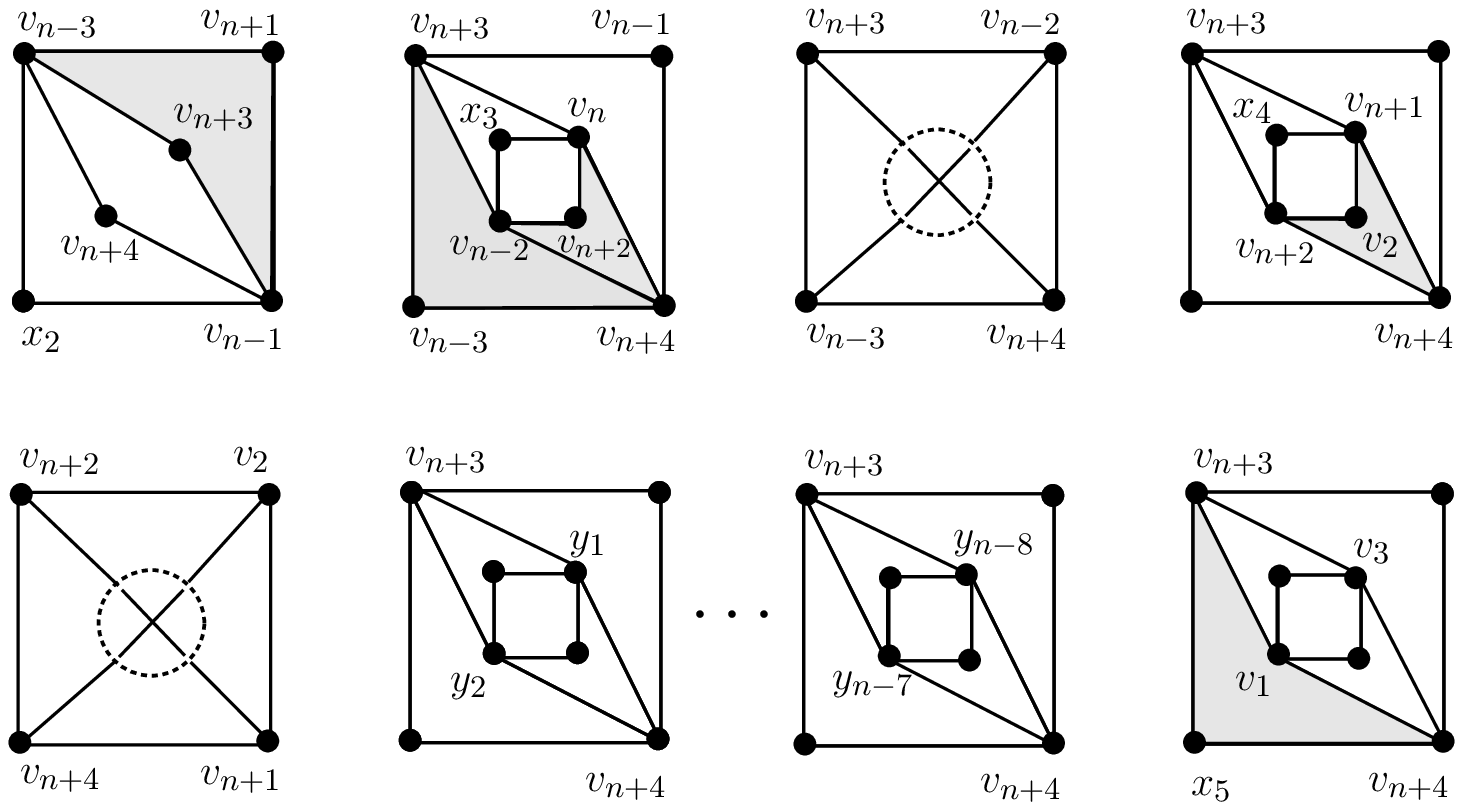} \\
  {Figure~\ref{nokn4}}
  \end{center}
  \end{figure}

 After the initial disc addition we have $\binom{n+2}2+2 = \binom{n+4}2
- (2n+3)$ edges.  Since handle additions add four edges and crosscap
additions add two edges, our process creates embeddings $\nQ(n+4,t)$ for
 $t= 2n+3, 2n-1, 2n-3, 2n-7, 2n-9, 2n-13, 2n-17, \dots, 5, 1$.
 However, we do not use the squares created by either crosscap addition in
any later handle addition, so we can omit one or both crosscap
additions.  If we omit the first crosscap addition, we obtain embeddings
$\nQ(n+4,t)$ for
 $t=2n+3, 2n-1, 2n-5, 2n-7, 2n-11, 2n-15, \dots, 7, 3$.
 Combining these, we have $\nQ(n+4,t)$ for $t = 2n+3, 2n-1, 2n-3, 2n-5,
\dots, n, \dots, 3, 1$, i.e., for all odd $t$ with $1 \le t \le 2n+3$
except $t = 2n+1$.  Since $n \ge 7$, this includes all $t \in
\nT_{n+4}$.

 If we perform all operations we obtain an embedding $\Phi_{n+4}$ of
$K_{n+4}-E_1$ where $E_1 = \{v_2 v_{n+3}\}$.
 If we perform all handle additions and the first crosscap addition but
omit the second crosscap addition, we obtain an embedding $\Phi^-_{n+4}$
of $K_{n+4}-E_3$, i.e., a $\nQ(n+4, 3)$.
 Observe the following.
 (a\pr) $E_3 = \{v_2 v_{n+3}, v_2 v_{n+4}, v_{n+1}
v_{n+2}\}$ ($E_1$ and the edges of the omitted crosscap addition).
 (b\pr) We see that $\Phi^-_{n+4}$ has a full diagonal set
 \begin{align*}
  \sD_{n+4} &= \{
	\di(v_1, v_3), \di(v_2, v_3),
	\di(y_1, y_2), \di(y_3, y_4), \dots, \di(y_{n-8}, y_{n-7}),
	\di(v_{n-3}, v_{n-1}), \di(v_{n-2}, v_n),
	\\ &\qquad
	\diq(v_{n+1}, v_{n+3}, v_{n+1} v_{n-1} v_{n+3} v_{n-3}),
	\diq(v_{n+2}, v_{n+4}, v_{n+2} v_n v_{n+4} v_{n-2})
 \}.
 \end{align*}
 (c\pr) There is a reserved square $v_1 v_{n+4} x_5 v_{n+3}$, or
equivalently $v_1 v_{n+3} x_5 v_{n+4}$, that is
not an underlying square for $\sD_{n+4}$.  Conditions (a\pr), (b\pr),
(c\pr) are almost what we need to say that $\Phi^-_{n+4}$ has Property
P.  Condition (b\pr) is correct but (a\pr) has $v_2$ where it should
have $v_1$, and (c\pr) has $v_1$ where it should have $v_2$.  However,
renaming $v_1$ as $v_2$ and $v_2$ as $v_1$ does not affect (b\pr), and
puts (a\pr) and (c\pr) into the correct form for Property P.  Thus,
after this renaming we have a $\nQ(n+4,3)$ with Property P.

 This completes the proof of the induction step.  Now the basis and the
induction step together imply Lemma \ref{nonoriodd}.
 \end{proof}

 Theorem \ref{construction} now follows from Lemmas \ref{orieven},
\ref{oriodd}, \ref{nonorismall}, \ref{nonorieven} and \ref{nonoriodd}.

 \section{Conclusion}\label{final}

 Face-simple quadrangulations are of interest because they are somewhere
in between closed $2$-cell embeddings and polyhedral embeddings.
 An embedding is \emph{closed $2$-cell} if every face is bounded by a
cycle, so that a face does not `self-touch' (equivalently,
a $2$-representative embedding of a $2$-connected graph), and
 \emph{polyhedral} if it is also true that two distinct faces touch at
most once, meaning that the intersection of their boundaries is empty, a
single vertex, or a single edge (equivalently, a $3$-representative
embedding of a $3$-connected graph).
 Every quadrangulation is closed $2$-cell by definition, but the
following lemma shows that minimal
quadrangulations cannot be polyhedral.

 \begin{lemma}
 If $\Phi$ is a quadrangular embedding of an $n$-vertex $m$-edge graph
with $m > \frac12 \binom{n}2$ then $\Phi$ is not polyhedral.
 \end{lemma}

 \begin{proof} Let $E$ be the edge set of the underlying (simple) graph
of $\Phi$, where each edge is considered as a vertex pair, and let $D$
be the multiset of diagonals of squares of $\Phi$, i.e., all vertex
pairs $\{u,v\}$ that occur as diagonals, counted by the number of
squares in which each diagonal occurs.  Then $|E|=|D|=m$
and since $|E|+|D| > \binom{n}2$ there is some pair that either occurs
twice in $D$, or occurs once in $E$ and once in $D$.  This means that
there are two faces that touch more than once.
 \end{proof}

 Therefore, it is natural to consider a weakening of polyhedral
to closed $2$-cell and face-simple, where two faces
can touch more than once, but not along two edges.
 In the orientable case, minimal quadrangulations of $S_g$, $g \ge 1$,
are automatically face-simple by Observation \ref{fsquad}, giving
Theorem \ref{face-simple}.  However, in general our nonorientable
minimal quadrangulations are not face-simple, since we often use
crosscap additions, which create two faces that touch
along two edges.

 \begin{question}
 Does $N_q$ have a minimal quadrangulation that is also face-simple, so
that $n'(N_q) = n(N_q)$, for all but a few small values of $q$?
 \end{question}

 We think that the answer is probably `yes.'  It should be possible to prove this by adapting the techniques in this
paper.  However, even if we avoid crosscap additions, some care is
needed.
 A handle addition of Type I (or a disc addition followed by a suitable
handle addition of Type I) preserves face-simplicity, but for
non\-ori\-ent\-able embeddings handle additions of Types II, III and IV
may create squares that touch along two edges.

 %
 %
 %
 %
 %


\begin{thebibliography}{0000}
 \parskip=-.5mm
  \def\jti#1{{\it\frenchspacing #1\/}} 
  \def\jvo#1{{\bf #1}} 

 \def\dbibitem#1{[#1]}



 \bibitem{BO} A. Bouchet, Orientable and nonorientable genus of the complete bipartite graph, \jti{J. Combin. Theory Ser. B}
 \jvo{24} (1978) 24-33.

 \bibitem{CR} D.L. Craft,
 On the genus of joins and compositions of graphs,
 \jti{Discrete Math.} \jvo{178} (1998) 25--50.

  \bibitem{GT} J.L. Gross and T.W. Tucker, \jti{Topological Graph
 Theory}, Dover, Mineola, New York, 2001.

  \bibitem{HA}
  N. Hartsfield, Nonorientable quadrangular embeddings of complete
 multipartite graphs,
  Proc. 25th Southeastern Internat. Conf. on Combinatorics, Graph Theory
 and Computing (Boca Raton, FL, 1994),
  \jti{Congr. Numer.} \jvo{103} (1994) 161--171.

 
 \bibitem{HA94p}
 Nora Hartsfield, The quadrangular genus of complete graphs, preprint,
1994.

 \bibitem{HR1} N. Hartsfield and G. Ringel, Minimal quadrangulations of
 orientable surfaces, \jti{J. Combin. Theory Ser. B} \jvo{46} (1989)
 84--95.

 \bibitem{HR2} N. Hartsfield and G. Ringel, Minimal quadrangulations of
 nonorientable surfaces, \jti{J. Combin. Theory Ser. A} \jvo{50} (1989)
 186--195.

  \bibitem{HU}
  J. P. Hutchinson, On coloring maps made from Eulerian graphs, Proc.
 Fifth British Combinatorial Conf. (Univ. Aberdeen, Aberdeen, 1975),
 \jti{Congr. Numer.} \jvo{15} (1976) 343--354.

   \bibitem{JR} M. Jungerman and G. Ringel, Minimum triangulations of
  orientable surfaces, \jti{Acta Math.} \jvo{145} (1980) 121--154.

 \bibitem{KO} V.P. Korzhik, Generating nonisomorphic quadrangular
embeddings of a complete graph, \jti{J. Graph Theory} \jvo{74} (2013)
133--142.

 \bibitem{KV} V.P. Korzhik and H-J. Voss, On the number of nonisomorphic
orientable regular embeddings of complete graphs, \jti{J. Combin. Theory
Ser. B} \jvo{81} (1) (2001) 58--76.

  \bibitem{LA}
  Serge Lawrencenko,
 Realizing the chromatic numbers and orders of spinal quadrangulations
of surfaces, \jti{J. Combin. Math. Combin. Comput.} \jvo{87} (2013)
303--308.
 
 \bibitem{LCY18}
 Serge Lawrencenko, Beifang Chen and Hui Yang, Determination of the
4-genus of a complete graph (with an appendix),
\texttt{arXiv:1803.03855v1}, 2018.

 \bibitem{LCYHp}
 Serge Lawrencenko, Beifang Chen, Hui Yang and Nora Hartsfield,
 The orientable 4-Genus Formula for the complete graph,
 in preparation.

  \bibitem{LEYZ} 
 W.  Liu,  S. Lawrencenko,  B. Chen,
  M.N. Ellingham, N. Hartsfield, H. Yang,
  D. Ye and X. Zha, Quadrangular embeddings of complete graphs and
  the Even Map Color Theorem,
 \jti{J. Combin. Theory Ser. B} \jvo{139} (1) (2019) 1--26.

 \bibitem{LC19}
 Shengxiang Lv and Yichao Chen, Constructing a minimum genus
embedding of the complete tripartite graph $K_{n,n,1}$ for odd $n$,
\jti{Discrete Math.} \jvo{342} (2019) 3017--3024.

 \bibitem{MR19}
 Dengju Ma and Han Ren, The orientable genus of the join of a
cycle and a complete graph, \jti{Ars Math. Contemp.} \jvo{17} (2019)
223--253.

 \bibitem{MMP} Z. Magajna, B. Mohar and T. Pisanski, Minimal ordered triangulations of surfaces, \jti{J. Graph Theory} \jvo{10} (1986) 451--460.

  \bibitem{MPP}
  B. Mohar, T. D. Parsons, and T. Pisanski, The genus of nearly complete
 bipartite graphs, \jti{Ars Combin.} \jvo{20} (1985) 173--183.

 \bibitem{PI1}
 T. Pisanski, Genus of Cartesian products of regular bipartite graphs,
 \jti{J. Graph Theory} \jvo{4} (1980) 31--42.

 \bibitem{PI2}
 T. Pisanski, Orientable quadrilateral embeddings of products of graphs,
 \jti{Discrete Math.} \jvo{109} (1992) 203--205.

 

 \bibitem{RI55} G. Ringel,  Wie man die geschlossenen nichtorientierbaren Fl\"{a}chen in m\"{o}glichst wenig
 Dreiecke zerlegen kann. \jti{Math. Ann. }  \jvo{130} (1955) 317--326.

 \bibitem{RI65a} G. Ringel, Das
Geschlecht des vollst\"{a}ndigen paaren Graphen,
 \jti{Abh. Math. Sem. Univ. Hamburg} \jvo{28} (1965) 139--150.

 \bibitem{RI65b} G. Ringel,  Der
vollst\"{a}ndige  paare Graph auf nichtorientierbaren  Fl\"{a}chen,
 \jti{ J. Reine Angew. Math.} \jvo{220} (1965) 88--93.

 \bibitem{RI} G. Ringel, \jti{Map Color Theorem}, Springer, Berlin, 1974.

 \bibitem{TH} C. Thomassen, The graph genus problem is NP-complete,
  \jti{J. Algorithms} \jvo{10} (1989) 568-576.

 \bibitem{THn} C. Thomassen, Triangulating a surface with a prescribed
graph, \jti{J. Combin. Theory Series B} \jvo{57} (1993) 196--206.

 \bibitem{WH70}
 A.T. White, The genus of repeated cartesian products of bipartite
graphs, \jti{Trans. Amer. Math. Soc.} \jvo{151} (1970) 393--404.

  \bibitem{WH}
  A.T. White,
  On the genus of the composition of two graphs,
  \jti{Pacific J. Math.} \jvo{41} (1972) 275--279.

 \end{thebibliography}
 \end{document}